\documentclass[a4paper]{article}
\date{September 2, 2025}
\usepackage[dvipdfmx]{graphicx}
\usepackage{latexsym}
\usepackage{amssymb}
\usepackage{amsmath}
\usepackage{amscd}
\usepackage{amsthm}
\usepackage{amsfonts}
\usepackage{mathrsfs}
\usepackage{enumerate}
\usepackage{enumitem}
\usepackage{bm}
\usepackage[usenames]{color}
\usepackage[active]{srcltx}
\usepackage{amsthm}
\theoremstyle{plain}
 \newtheorem{theorem}{Theorem}[section]
 \newtheorem{proposition}[theorem]{Proposition}
 
 \newtheorem*{fact*}{Fact}
 \newtheorem{lemma}[theorem]{Lemma}
 \newtheorem{corollary}[theorem]{Corollary}
 \theoremstyle{remark}
 \newtheorem{definition}[theorem]{Definition}
 \newtheorem{remark}[theorem]{Remark}
 
 \newtheorem{example}[theorem]{Example}
\numberwithin{equation}{section}

\newcommand{\R}{\boldsymbol{R}}
\newcommand{\vect}[1]{\boldsymbol{#1}}

\begin{document}
\title{Isometric deformations of pillow boxes}

\author{
Atsufumi Honda\footnote{Corresponding author.}\\
\vspace{2mm}
{\it\small Department of Applied Mathematics, Yokohama National University,}
\vspace{-2mm}\\
{\it\small Hodogaya, Yokohama 240-8501, Japan}\\
{\small\tt honda-atsufumi-kp@ynu.ac.jp}\\
{\small https://orcid.org/0000-0003-0515-2866}
\vspace{6mm}\\
Miyuki Koiso\\
\vspace{2mm}
{\small\it Institute of Mathematics for Industry, Kyushu University,}
\vspace{-2mm}\\
{\it\small Motooka Nishi-ku, Fukuoka 819-0395, Japan}\\
{\small\tt koiso@math.kyushu-u.ac.jp}\\
{\small https://orcid.org/0000-0002-3686-7277}
}

\maketitle

\begin{abstract}
Pillow boxes are surfaces used for gift boxes, packaging, and even architectural applications.
By definition, a pillow box is isometric to a double rectangle consisting of two copies of a rectangle.
If the crease pattern is allowed to change, 
there exist continuous isometric deformations 
from a pillow box
to a double rectangle.
However, practical applications often require preserving the crease pattern.
In this paper, 
we classify isometric deformations from a pillow box to a double rectangle 
among curved foldings that preserve the crease pattern.
As a corollary, 
we prove that such an isometric deformation 
necessarily changes the topology of a pillow box.
\end{abstract}

\renewcommand{\thefootnote}{\fnsymbol{footnote}}
\footnote[0]{2020 {\it Mathematics Subject Classification.}
Primary 53A05; 
Secondary 51M15. 
} 
\footnote[0]{{\it Key Words and Phrases.} 
curved folding,
origami,
developable surface,
pillow box,
isometric deformation,
first fundamental form.
} 

\section{Introduction}

\begin{figure}[htb]
\centering
 \begin{tabular}{c@{\hspace{18mm}}c@{\hspace{5mm}}c}
  \resizebox{5.0cm}{!}{\includegraphics{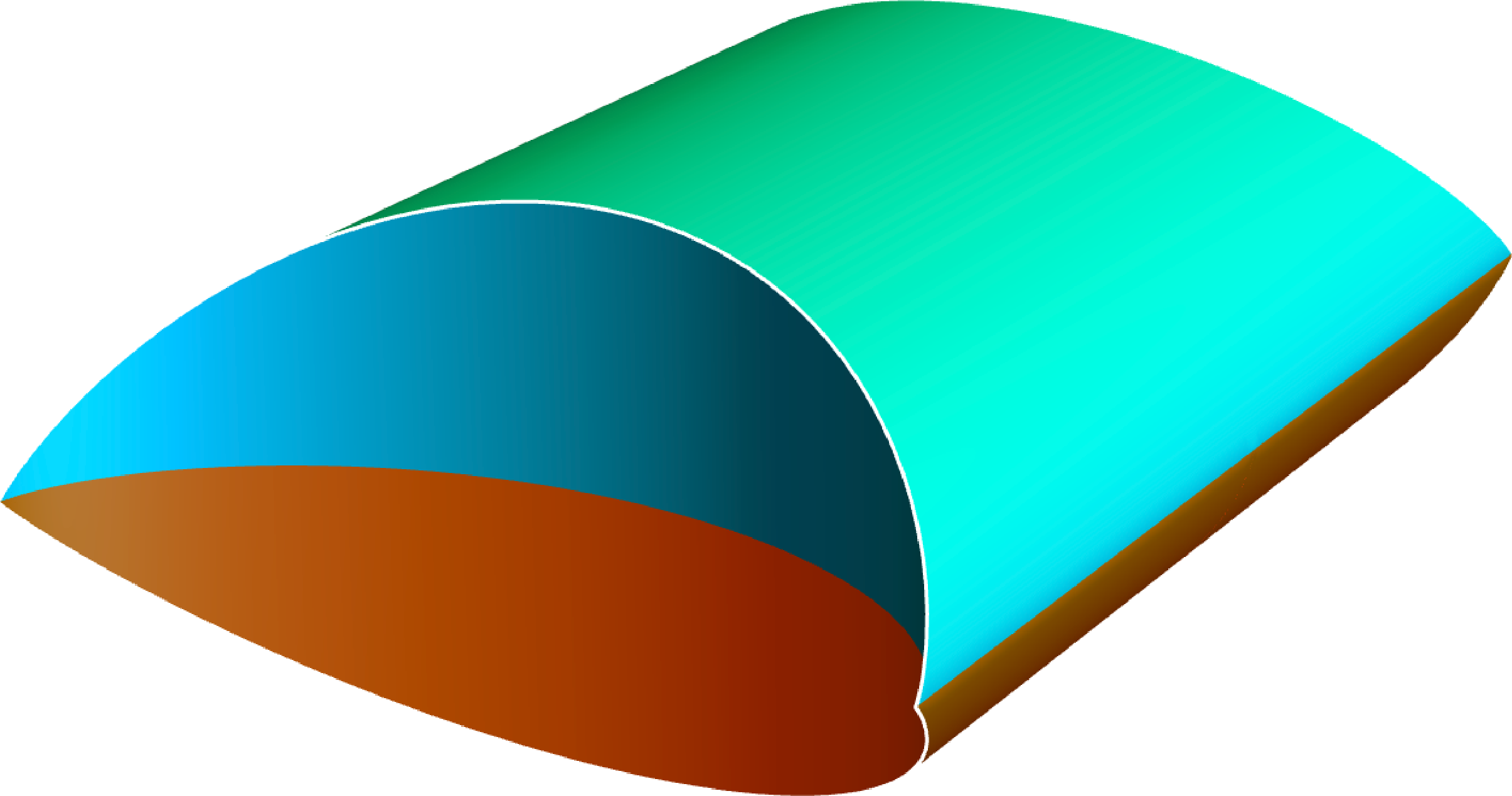}}&
  \resizebox{2.5cm}{!}{\includegraphics{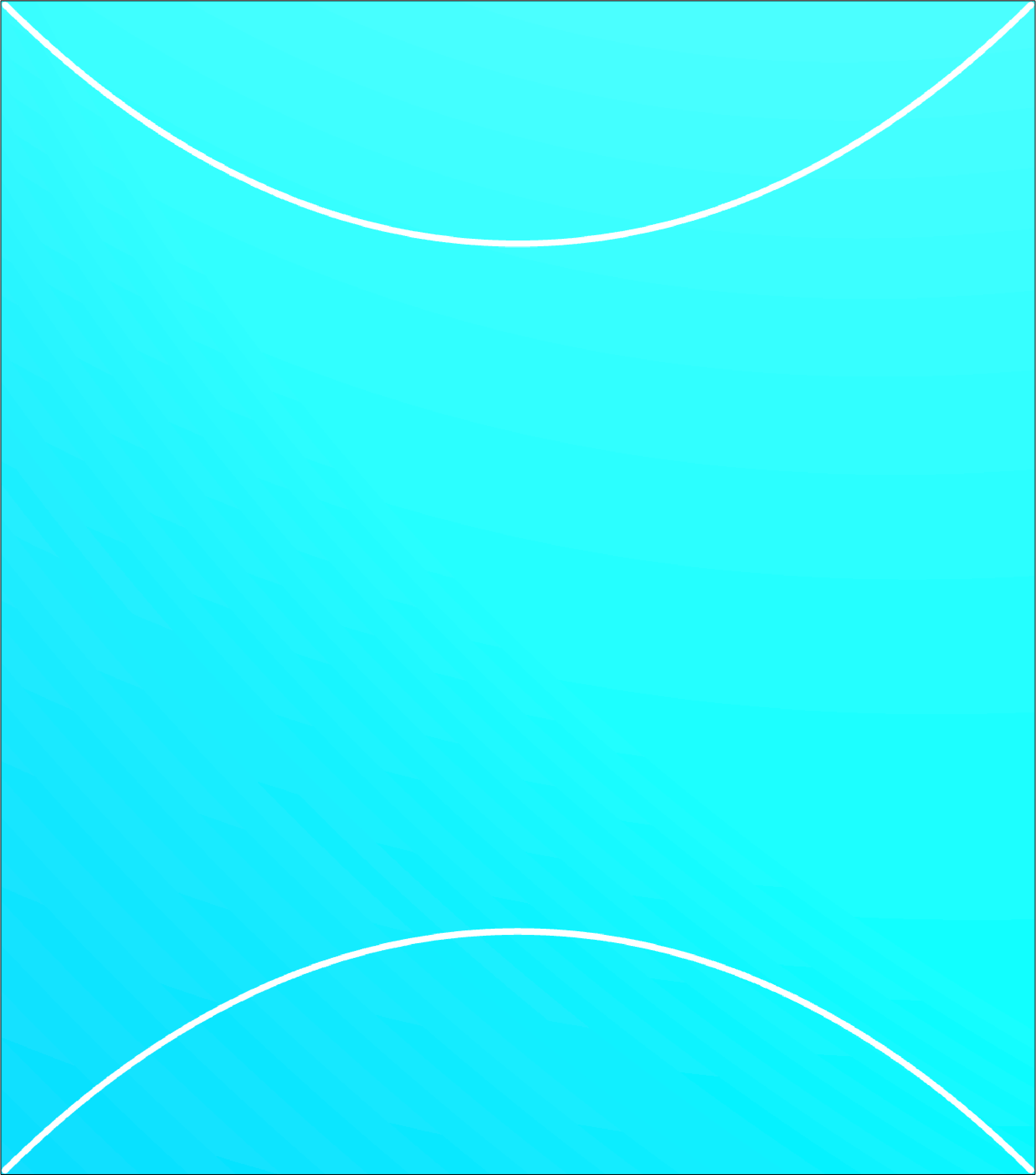}}&
  \resizebox{2.5cm}{!}{\includegraphics{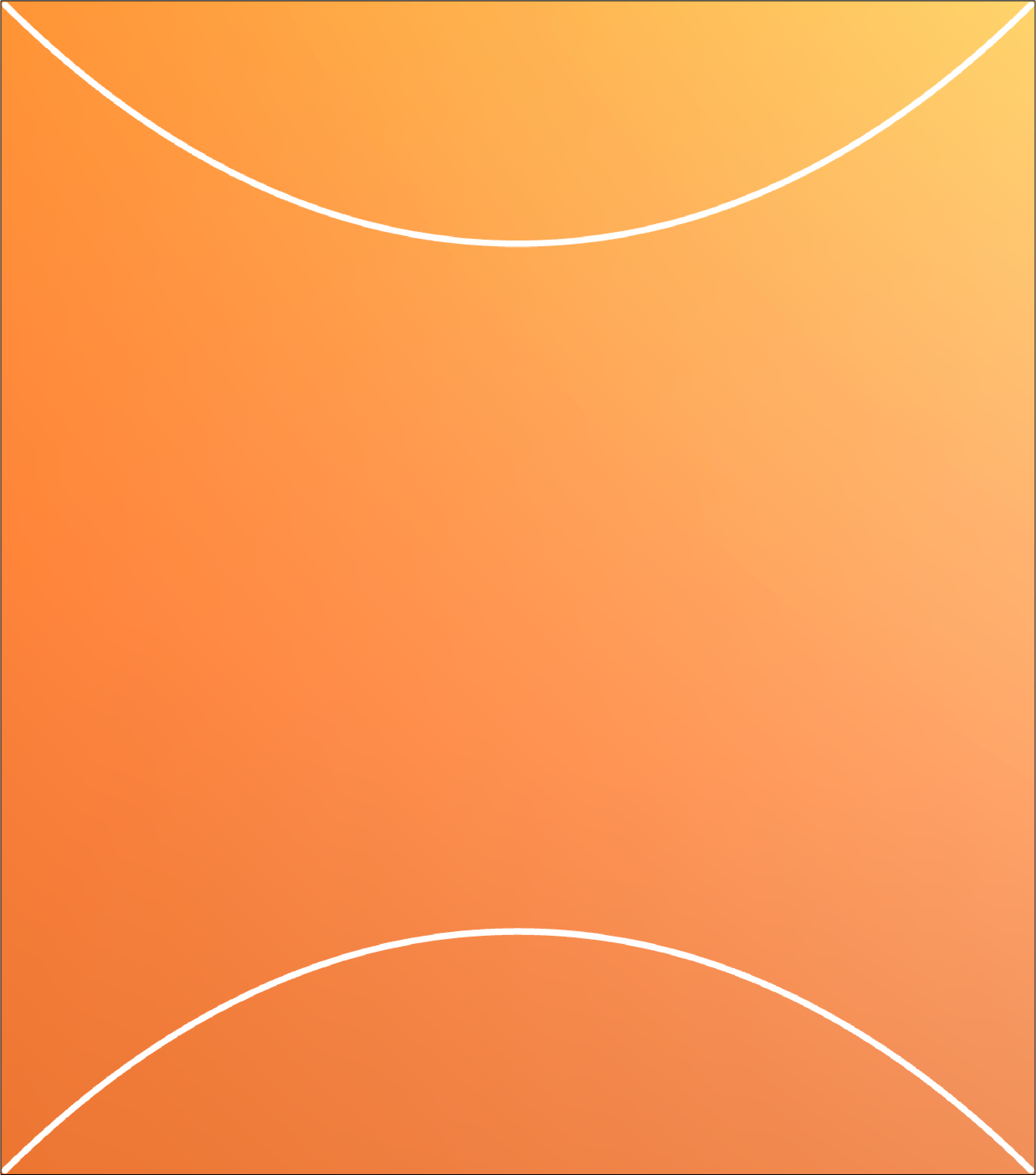}}
 \end{tabular}
  \caption{A pillow box (left) is a surface created by 
  folding two copies of a rectangle (right). 
  See Definition \ref{def:pillow-box},
  cf.\ Subsection \ref{sec:folding-rectangles}.}
\label{fig:rectangle-pillow}
\end{figure}

A {\it pillow box} is a piecewise flat surface constructed 
by a curved folding (i.e., origami) of two copies of a rectangle 
(cf.\ Figure \ref{fig:rectangle-pillow}; see also Definition \ref{def:pillow-box} for details).
Pillow boxes are commonly used as gift boxes and 
packaging materials, and thus have practical applications.
Koiso \cite{Koiso} proved that
there uniquely exists a pillow box that maximizes the enclosed volume,
and its base curve is an elastic curve
(cf.\ Figure \ref{fig:planes-quarter}).
Mitani \cite{Mitani} discusses the design and 
the enclosed volume for specific examples of pillow boxes.

Recently, pillow boxes have also been used in the construction of 
temporary housing  
or temporary shelters
(cf.\ Tachi \cite{Tachi}).
Producing temporary housing entails deforming a rectangular sheet of paper into a pillow box without stretching or shrinking.
Such a process is mathematically referred to as an {\it isometric deformation}.

In this paper, we study isometric deformations of pillow boxes.
By definition, a pillow box is isometric to a {\it double rectangle}, 
which consists of two copies of a rectangle 
(see Definition~\ref{def:double-rectangle} for details).
If the crease pattern is allowed to change, it is possible to construct a continuous isometric deformation from a pillow box to a double rectangle (cf.\ Figure~\ref{fig:pillow-0-deformation-long}).
However, in practical applications, isometric deformations that preserve the crease pattern are preferable.
We classify isometric deformations from a pillow box to a double rectangle among curved foldings that preserve the crease pattern (Theorem~\ref{thm:Koiso-deformation}).
As a corollary, we prove that 
such an isometric deformation 
necessarily changes the topology of a pillow box
(Corollary~\ref{cor:isometric-deformation}).

\begin{figure}[htb]
\centering
 \begin{tabular}{c@{\hspace{4mm}}c@{\hspace{4mm}}c}
  \resizebox{4.0cm}{!}{\includegraphics{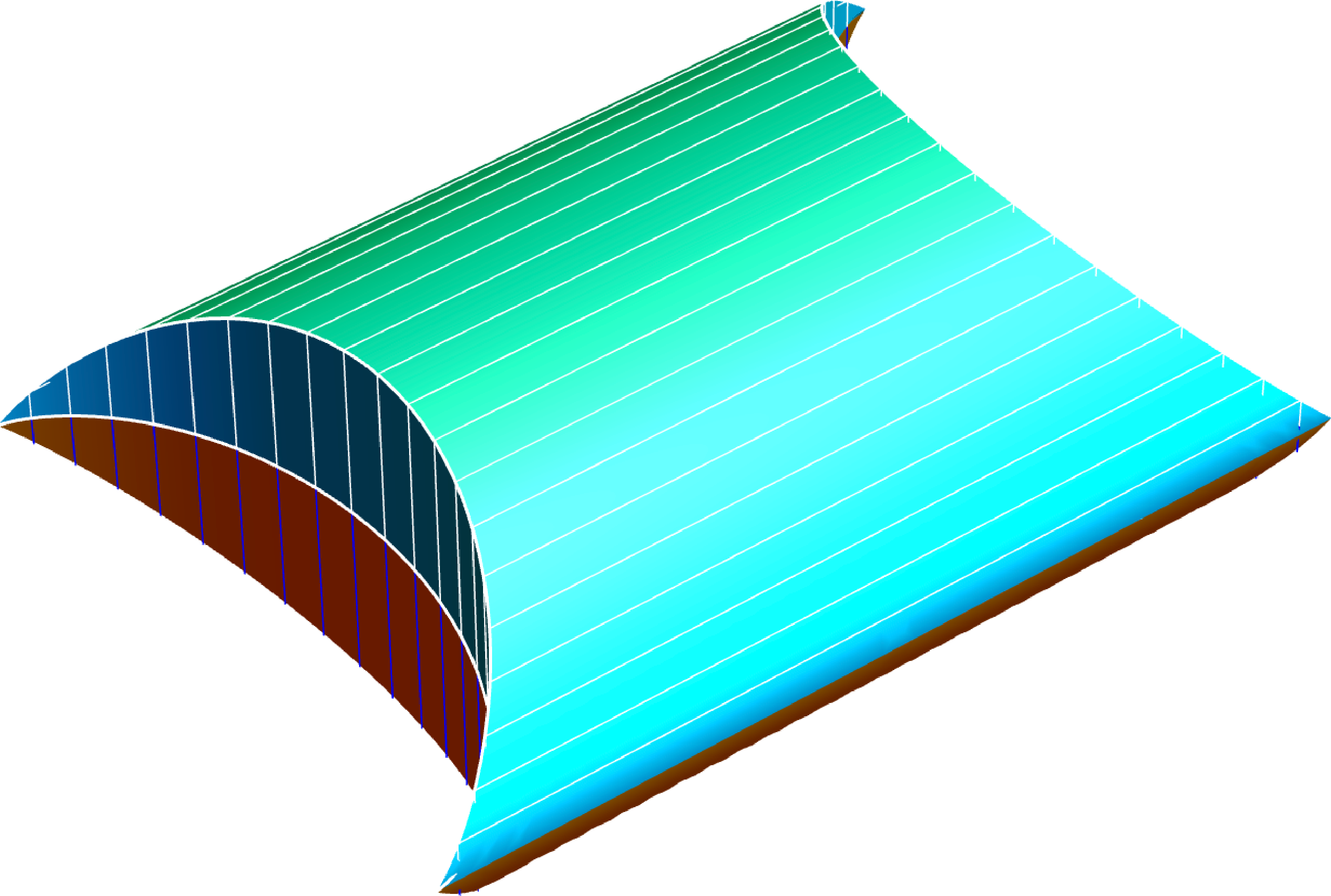}}&
  \resizebox{4.0cm}{!}{\includegraphics{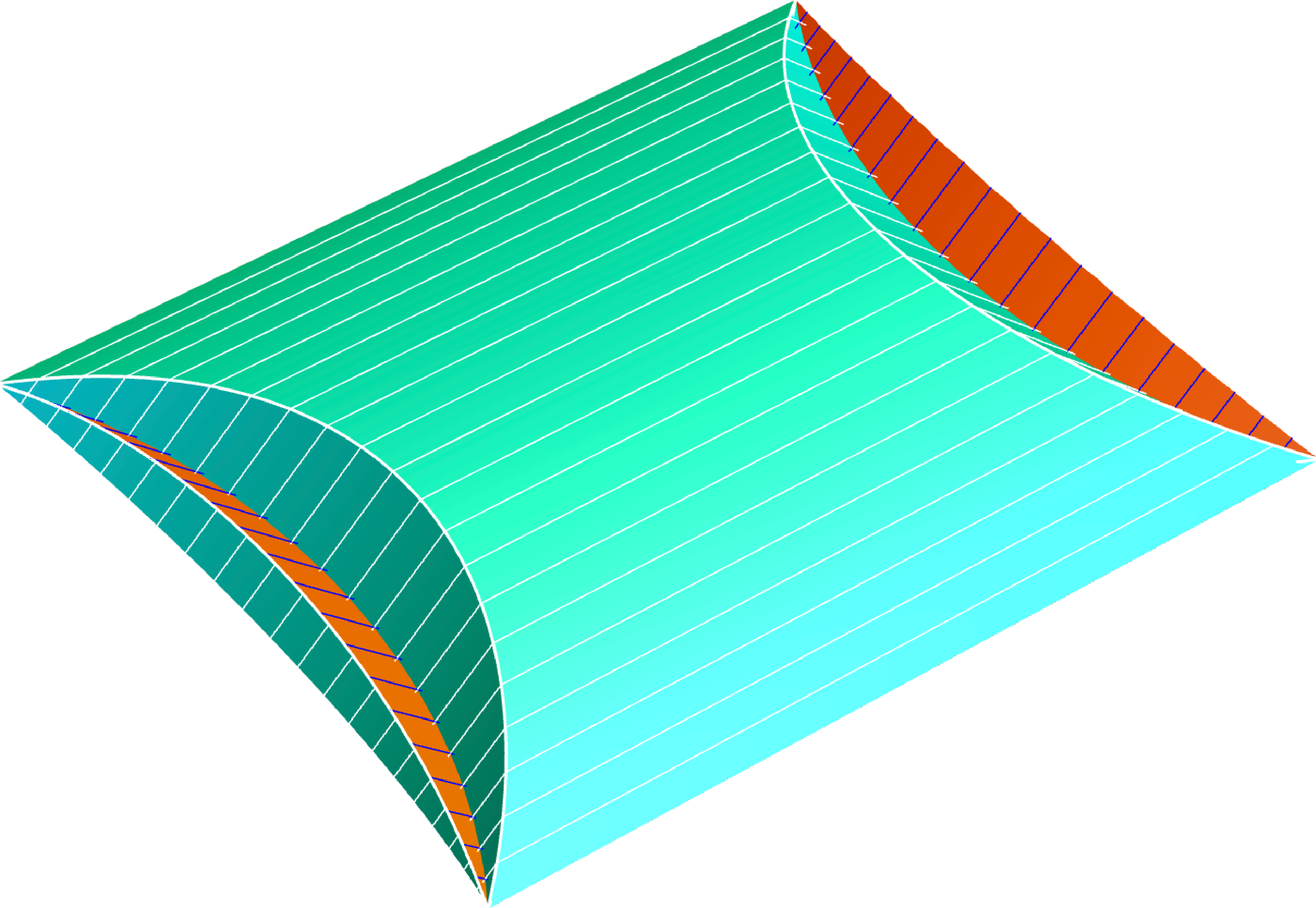}}&
  \resizebox{4.0cm}{!}{\includegraphics{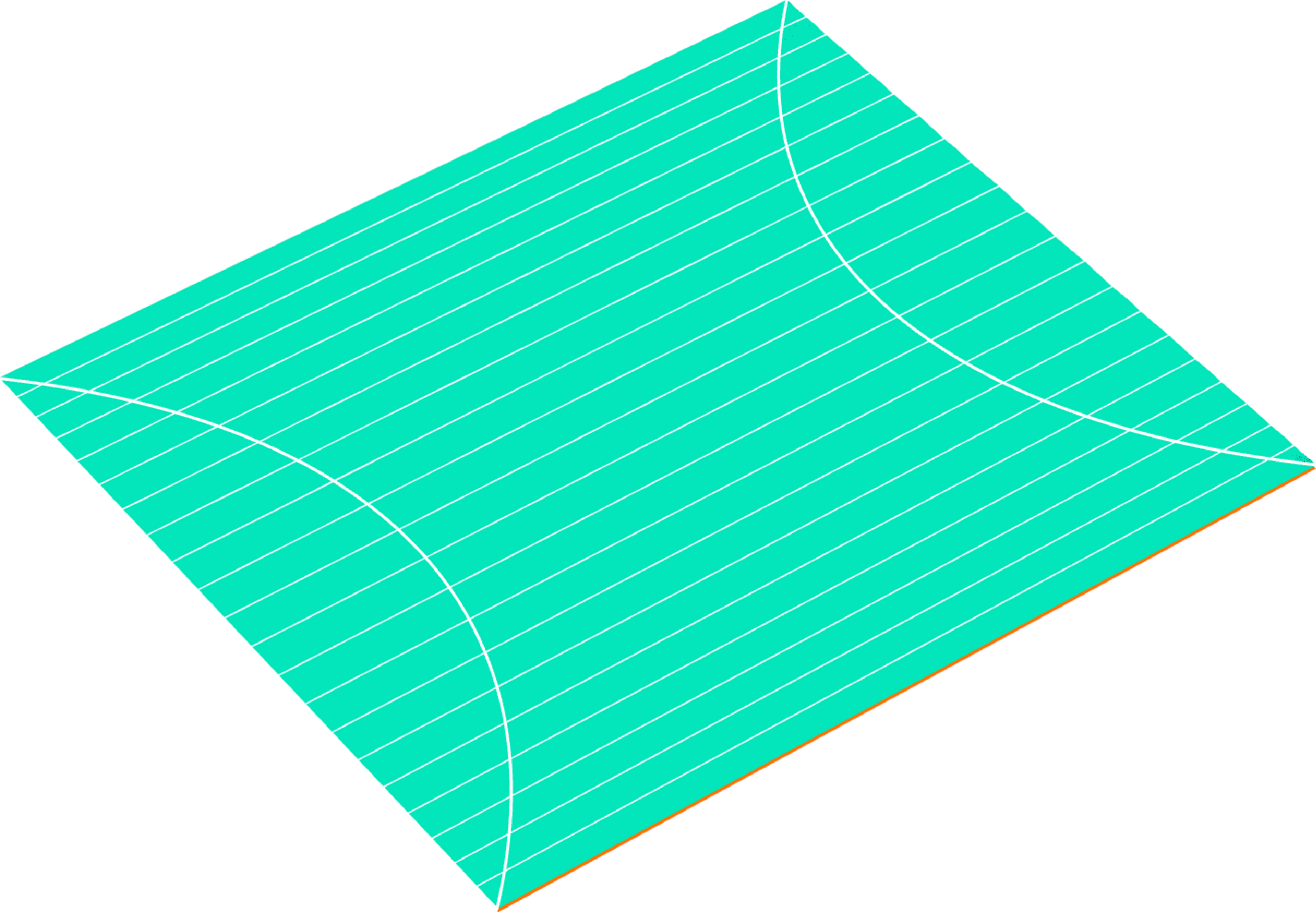}}\\
  {\footnotesize (a) $t=0$}&
  {\footnotesize (b) $t=0.6$}&
  {\footnotesize (c) $t=1$}
 \end{tabular}
 \caption{An isometric deformation of a pillow box (a) to a double rectangle (c) that preserves the crease pattern.
See Example~\ref{ex:Koiso-deformation}.
We may observe that both (a) and (c) are homeomorphic to a sphere, whereas (b) is not.}
\label{fig:pillow-deformation}
\end{figure}

Another motivation for this study is the so-called ``bellows conjecture''.
The bellows conjecture states that the enclosed volume of a polyhedron is invariant under isometric deformations.
This conjecture was formulated 
by R.~Connelly and D.~Sullivan (cf.\ \cite{Connelly, Ghys}), 
and was proved by Sabitov \cite{Sabitov, Sabitov2} and 
Connelly, Sabitov and Walz \cite{CSW}.
Our result suggests that the bellows conjecture might hold for a wider class of surfaces.

This paper is organized as follows.
In Section~\ref{sec:prelim}, we review the basic concepts of curved foldings.
In Section~\ref{sec:Pillow-boxes}, we introduce the definition of pillow boxes and observe that the quarter halves of pillow boxes are curved foldings.
In Section~\ref{sec:origami-deformation}, we define origami deformations, which are isometric deformations from pillow boxes to double rectangles among curved foldings that preserve the crease pattern (Definition~\ref{def:origami-deformation}).
We then obtain a representation formula for origami deformations (Theorem~\ref{thm:Koiso-deformation}).
As a corollary, we prove that such an isometric deformation does not preserve the topology of a pillow box (Corollary~\ref{cor:isometric-deformation}).

\section{Preliminaries}
\label{sec:prelim}

In this section, 
we review the basic concepts of curved foldings.
For details, see \cite{HNSUY3} (cf.\ \cite{FT0, HNSUY4, Hu, KFCMSP}).
We let $\vect{c} : J \to \R^3$ be an injective regular curve
parametrized by arc length,
where $J$ is a non-empty closed bounded interval.
We call $T(s) := \vect{c}'(s)$ 
the {\it unit tangent vector field} along $C := \vect{c}(J)$,
where the prime denotes differentiation with respect to $s$.
Suppose that the {\it curvature function} 
$\kappa(s) := \|T'(s)\|$ is positive on $J$.
Here, we set 
$\|\vect{x}\| := \sqrt{\vect{x} \cdot \vect{x}}$
for $\vect{x} \in \R^3$, where ``$\cdot$'' denotes the canonical inner product of $\R^3$.
The {\it principal normal vector field} $N(s)$ and 
the {\it binormal vector field} $B(s)$ are defined by
\[
N(s) := \frac{1}{\kappa(s)} T'(s), 
\quad
B(s) := T(s) \times N(s),
\]
respectively, where ``$\times$'' denotes the vector product in $\R^3$.
Then $\tau(s) := N'(s) \cdot B(s)$ 
is called the {\it torsion function}.

Let $\alpha(s)$ and $\beta(s)$ ($s \in J$)
be smooth functions satisfying
\begin{equation}\label{eq:angle}
0 < |\alpha(s)| < \frac{\pi}{2}, \quad 
0 < \beta(s) < \pi
\qquad (s \in J).
\end{equation}
We set 
\begin{align*}
&p(s, v) := \vect{c}(s) + v\, \xi(s)\\
&
\Bigl(\xi(s) :=
\cos \beta(s) T(s) +
\sin \beta(s) \bigl(
\cos \alpha(s) N(s) + \sin \alpha(s) B(s) \bigr) \Bigr).
\end{align*}
Since $\vect{c}'(s) = T(s)$ and $\xi(s)$
are linearly independent, there exist
smooth functions $\varepsilon(s) > 0$ and $\delta(s) < 0$ on $J$
such that the ruled surface $p : U \to\R^3$
is an embedding on 
\[
U := 
\left\{ (s, v) \in \R^2 \mid
s \in J,~ \delta(s) \leq v \leq \varepsilon(s) \right\}.
\]
The functions $\alpha(s)$ and $\beta(s)$ are 
called the {\it first angular function} and the 
{\it second angular function}
of $p(s, v)$, respectively.
Then, $p(s, v)$ is developable 
--- i.e., its Gaussian curvature $K$ is identically zero --- 
if and only if
\begin{equation}\label{eq:beta}
\cot \beta(s) = \frac{\alpha'(s) + \tau(s)}{\kappa(s) \sin \alpha(s)}
\end{equation}
holds.
In particular, $p(s, v)$ is uniquely determined 
by the first angular function $\alpha(s)$.
Such $p(s, v)$ is called a {\it developable surface along $\vect{c}(s)$}.

We set 
\begin{align}
\label{eq:conormal}
\nu(s) &:= -\sin \alpha(s) N(s) + \cos \alpha(s) B(s),\\
\label{eq:normal-p}
\vect{n}_g(s) &:= \cos \alpha(s) N(s) + \sin \alpha(s) B(s),
\end{align}
which give
the unit normal vector field of $p(s, v)$
and the unit conormal vector field along $\vect{c}(s)$, 
respectively.
Then, 
\begin{equation}\label{eq:kappa-g}
\kappa_g(s) := T'(s) \cdot \vect{n}_g(s) = \kappa(s) \cos \alpha(s)
~(>0)
\end{equation}
is the geodesic curvature of $\vect{c}(s)$ as a curve on $p(s, v)$,
which is positive by \eqref{eq:angle}.

Let $\gamma : J \to \R^2$ be a plane curve 
whose curvature function coincides with $\kappa_g(s)$.
Since $p : U \to\R^3$ is a developable surface,
it can be developed onto a plane $\R^2$.
Since a developing map is an isometry,
the space curve $C = \vect{c}(J)$ 
corresponds to the plane curve $\Gamma := \gamma(J)$.
Then, curved foldings are defined as follows.

\begin{definition}\label{def:curved-folding}
Let $p, q : U \to\R^3$ be developable surfaces along 
a space curve $\vect{c}(s)$ with positive curvature.
Suppose that 
the first angular functions of $p(s, v)$ and $q(s, v)$
are given by $\alpha(s)$ and $-\alpha(s)$,
respectively.
$($Such a developable surface $q(s, v)$
is called the {\it dual} of $p(s, v).)$
We set $U=U_+\cup U_-$, 
where $U_{\pm}:=\{ (s, v) \in U \mid \pm v \geq 0 \}$.
Then, we call a map $X : U \to\R^3$ defined by
\[
X(s, v) :=
\begin{cases}
p(s, v) & ((s, v) \in U_+), \\
q(s, v) & ((s, v) \in U_-),
\end{cases}
\]
the {\it origami-map} associated with
the developable surfaces $p$ and $q$.
The image 
\[
P := X(U)
\]
of the origami map $X$ 
is called a {\it curved folding}.
The space curve $C = \vect{c}(J)$
and the plane curve $\Gamma = \gamma(J)$
are called 
the {\it crease} and the {\it crease pattern} of 
the curved folding $P$,
respectively.
\end{definition}

\section{Pillow boxes}
\label{sec:Pillow-boxes}

In this section,
we introduce the definition of pillow boxes
(Definition~\ref{def:pillow-box}).
Quarter halves of pillow boxes are shown to be curved foldings
(Proposition~\ref{prop:origami-parametrization}).
Via the developing maps of developable surfaces,
we see that a pillow box is isometric to a double rectangle 
consisting of two copies of a rectangle.
A procedure for constructing a pillow box from a double rectangle 
is also presented.

\subsection{Definition of pillow boxes}
\label{sec:quarter-domain}

We fix positive numbers $b, d > 0$.
Let $f : [0, d] \to \R$ be a smooth function 
which satisfies
\begin{equation}\label{eq:function-f}
f(0) = f(d) = 0, \qquad
0 < f(x) < b, \quad
f''(x) < 0
\quad
(0 < x < d).
\end{equation}
We set 
\[
P := P_+ \cup P_-,
\]
where
\begin{equation}\label{eq:P+-}
\begin{split}
P_+ &:= \{ (x, y, f(x)) \in \R^3 \mid 0 \leq x \leq d,\; f(x) \leq y \leq b \},\\
P_- &:= \{ (x, f(x), z) \in \R^3 \mid 0 \leq x \leq d,\; 0 \leq z \leq f(x) \}.
\end{split}
\end{equation}
We denote by $\rho_V$ and $\rho_H$
the reflections with respect to the planes $y = b$ and $z = 0$,
which are defined by 
\begin{equation}\label{eq:reflection}
  \rho_V(x, y, z) := (x, 2b - y, z), \qquad
  \rho_H(x, y, z) := (x, y, -z)
\end{equation}
for each $(x, y, z) \in \R^3$.

\begin{definition}\label{def:pillow-box}
We call the union
\begin{equation}\label{eq:M-entire}
M := P \cup \rho_V(P) \cup \rho_H(P) \cup \rho_H \circ \rho_V(P)
\end{equation}
a {\it pillow box}.
See Figure~\ref{fig:rectangle-pillow}, left.
Each of $P$, $\rho_V(P)$, $\rho_H(P)$, and $\rho_H \circ \rho_V(P)$
is called a {\it quarter domain} of $M$.
A pillow box $M$ is symmetric with respect to 
the planes $y = b$ and $z = 0$.
The plane $y = b$ is called 
the {\it vertical plane} of the pillow box $M$,
which is denoted by $V_M$,
and the plane $z = 0$ is called the {\it horizontal plane},
which is denoted by $H_M$.
See Figure~\ref{fig:planes-quarter}.
\end{definition}

By definition, a pillow box $M$
is homeomorphic to the $2$-dimensional sphere $S^2$.

\begin{figure}[htb]
\centering
 \begin{tabular}{c@{\hspace{2mm}}c@{\hspace{5mm}}c}
  \resizebox{5cm}{!}{\includegraphics{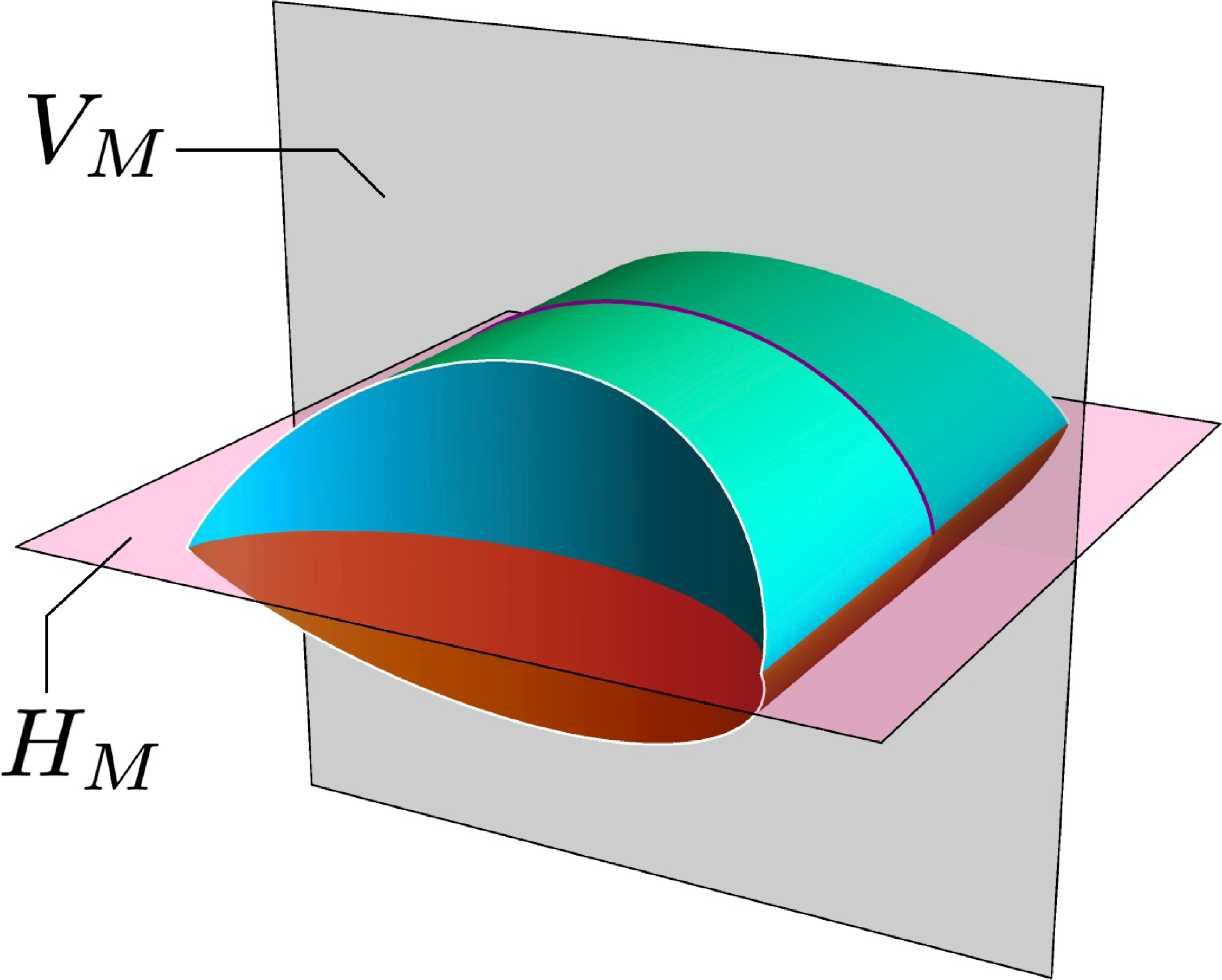}}&
  \resizebox{4.0cm}{!}{\includegraphics{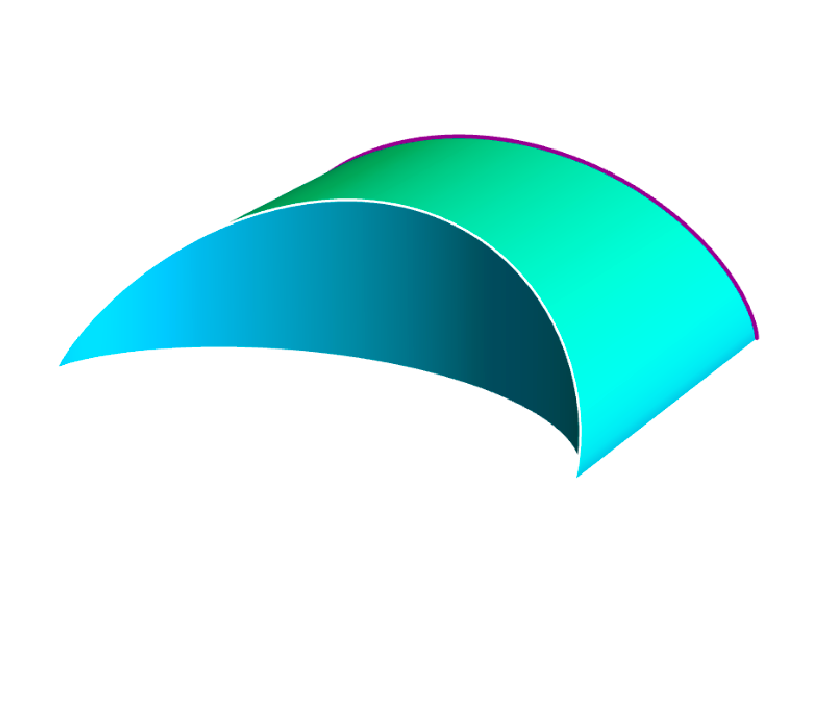}}&
  \resizebox{3.5cm}{!}{\includegraphics{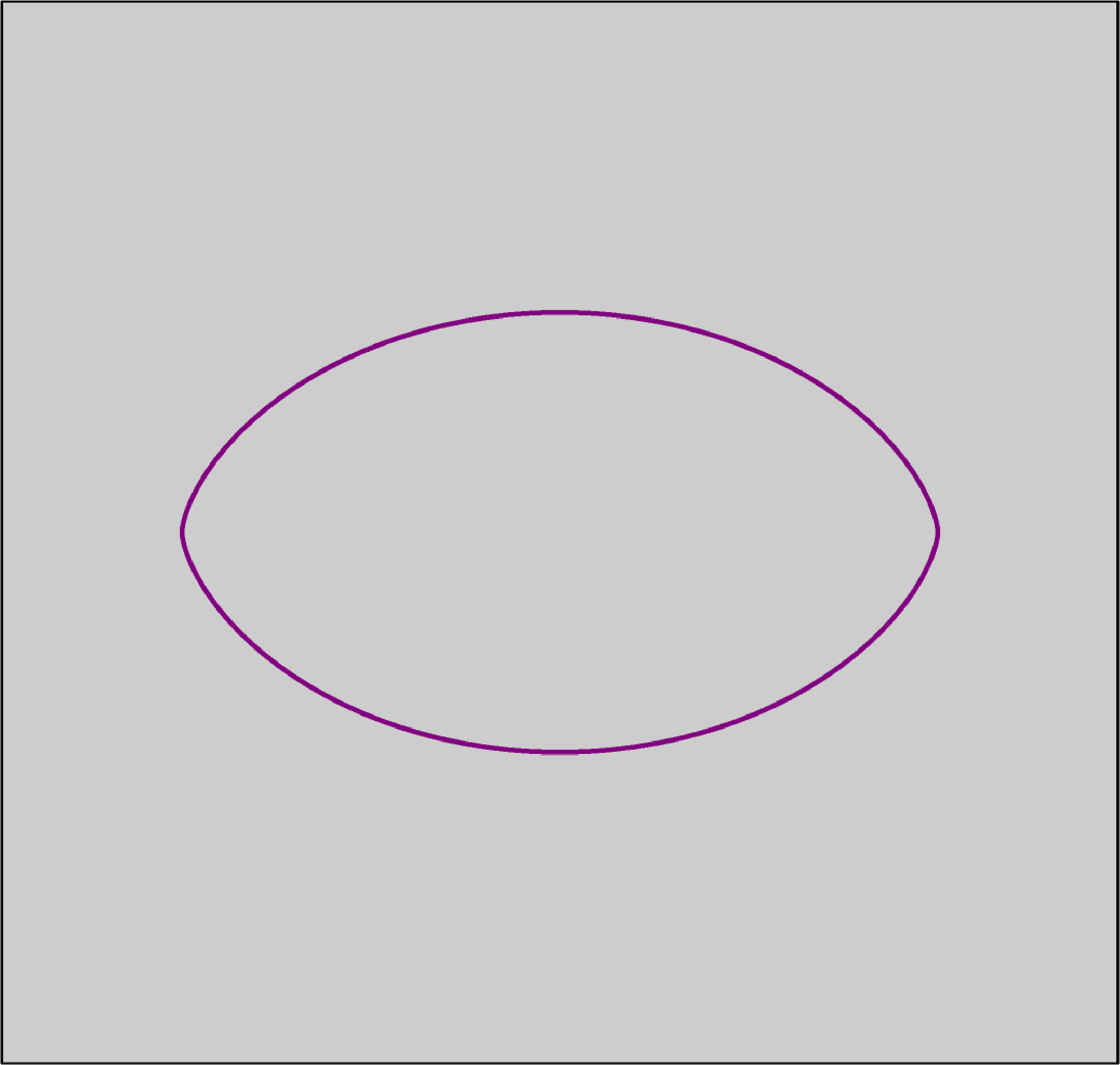}}
 \end{tabular}
 \caption{Left: 
a pillow box $M$ with its horizontal and vertical planes, $H_M$ and $V_M$.
Center: 
a quarter domain $P$ of the pillow box $M$.
Right:
the intersection of a pillow box $M$ and its vertical plane $V_M$
is the union of two regular curves, which is called the {\it base curve}.
The base curve is parametrized as $(x, \pm f(x))$ in the vertical plane.
}
\label{fig:planes-quarter}
\end{figure}

\subsection{Pillow boxes as curved foldings}
\label{sec:pillow-box-origami}

The intersection 
\[
C := P_+ \cap P_-
= \{ (x, f(x), f(x)) \in \R^3 \mid 0 \leq x \leq d \}
\]
is a regular curve embedded in $\R^3$
of finite length $L$.
We define $\vect{c} : [0, L] \to \R^3$ as
\begin{equation}\label{eq:crease}
\vect{c}(s) := 
\left( \int_{0}^s \sigma(w) \, dw,\, \zeta(s),\, \zeta(s) \right)
\qquad
\left( \sigma(s) := \sqrt{1 - 2\, \zeta'(s)^2} \right),
\end{equation}
which is an arc-length parametrization of $C$.
Since $x=\int_{0}^s \sigma(w) \, dw$ 
is a parameter change, $\sigma(s)\ne0$ holds,
and hence, we have
\begin{equation}\label{eq:zeta-prime}
1 - 2\, \zeta'(s)^2 > 0.
\end{equation}
Moreover, \eqref{eq:function-f} holds if and only if 
\begin{equation}\label{eq:boundary-condition}
\zeta(0) = \zeta(L) = 0,
\qquad
0 < \zeta(s) < b,
\qquad
\zeta''(s) < 0
\end{equation}
hold for each $s \in (0, L)$.
Let $p, q : U \to\R^3$ 
be developable surfaces defined by
\begin{equation}\label{eq:pq}
p(s, v)
:= \vect{c}(s) + v\, \xi, \qquad
q(s, v)
:= \vect{c}(s) + v\, \check{\xi},
\end{equation}
where 
\begin{gather}
\label{eq:D}
U := \{ (s, v) \in \R^2 \mid 0 \leq s \leq L,\, \zeta(s) - b \leq v \leq \zeta(s) \},\\
\label{eq:xi-xi-hat}
\xi :=(0, 0, -1), \qquad
\check{\xi} := (0, -1, 0).
\end{gather}
We define a map $X : U \to\R^3$ as
\begin{equation}\label{eq:pillow-origami}
X(s, v) := 
\begin{cases} 
p(s, v) & ((s, v) \in U_+),\\
q(s, v) & ((s, v) \in U_-),
\end{cases}
\end{equation}
where
\begin{align}
\label{eq:D+}
U_+ &:= \{ (s, v) \in \R^2 \mid 0 \leq s \leq L,\, 0 \leq v \leq \zeta(s) \},\\
\label{eq:D-}
U_- &:= \{ (s, v) \in \R^2 \mid 0 \leq s \leq L,\, \zeta(s) - b \leq v \leq 0 \}.
\end{align}
Then, the image $X(U)$ is 
a quarter domain $P$ of the pillow box $M$
given in \eqref{eq:M-entire}.
Then, the following holds:

\begin{proposition}\label{prop:origami-parametrization}
The map $X : U \to\R^3$ defined by \eqref{eq:pillow-origami}
is an origami-map such that $X$ 
is a homeomorphism onto its image $P = X(U)$. 
$($We call $X$ the {\it origami-parametrization} 
of $P$.$)$
\end{proposition}

\begin{proof}
By the definition of the map $X$ in \eqref{eq:pillow-origami},
the injectivity of $X$ can be shown directly.
Since $X$ is continuous and $U$ is compact, 
it follows that $X$ is a homeomorphism between $U$ and $P = X(U)$.
So, it suffices to show that $X$ is an origami-map.

By \eqref{eq:crease}, 
the unit tangent vector 
$T(s) = \vect{c}'(s)$ is written as
$T = (\sigma,\, \zeta',\, \zeta')$.
Since 
$T' = \zeta''(-2\, \zeta'/\sigma,\, 1,\, 1)$,
the curvature function $\kappa(s)$,
the principal normal $N(s)$,
and the binormal $B(s)$
are given by
\[
\kappa = -\sqrt{2}\, \zeta'' / \sigma,\quad
N = (2\, \zeta',\, -\sigma,\, -\sigma)/\sqrt{2},\quad
B = (0,\, 1,\,-1)/\sqrt{2}.
\]

The vector field 
\[
\vect{n}_g(s)
=
\bigl( \sigma(s) \zeta'(s),\, \zeta'(s)^2,\, -1 + \zeta'(s)^2 \bigr)/\sqrt{1 - \zeta'(s)^2},
\]
gives the unit conormal vector field 
on the developable surface $p(s,v)$
along the curve $\vect{c}(s)$.
Since (cf.\ \eqref{eq:normal-p})
$$
\cos\alpha(s) 
= \sigma(s)/\sqrt{2 - 2\, \zeta'(s)^2},\qquad
\sin\alpha(s) 
= 1/\sqrt{2 - 2\, \zeta'(s)^2},
$$
the first angular function $\alpha(s)$ of $p(s, v)$ is given by
$\alpha(s) = \arctan \bigl( 1 / \sigma(s) \bigr).$
The second angular function $\beta(s)$ of $p(s, v)$ satisfies
\begin{equation}\label{eq:beta-p}
\cos \beta(s) = \xi \cdot T(s) = - \zeta'(s), \quad
\sin \beta(s) = \xi \cdot \vect{n}_g(s) = \sqrt{1 - \zeta'(s)^2}.
\end{equation}
Similarly, since
$
\check{\vect{n}}_g
=
( \sigma \zeta',\, -1 + (\zeta')^2,\, (\zeta')^2 )/\sqrt{1 - (\zeta')^2}
$
is the conormal vector field on $q(s, v)$ along the curve $\vect{c}(s)$,
we obtain that
the first angular function $\check{\alpha}(s)$ of $q(s, v)$
is $\check{\alpha}(s) = -\alpha(s)$,
and the second angular function $\check{\beta}(s)$ of $q(s, v)$
coincides with $\beta(s)$.
This proves the assertion.
\end{proof}

\begin{definition}\label{def:F-data}
Let $b$ be a positive real number,
and let $\zeta(s)$ $(0 \leq s \leq L)$ 
be a smooth function satisfying 
\eqref{eq:zeta-prime} and \eqref{eq:boundary-condition}.
The pair $(b, \zeta(s))$
is called {\it fundamental data}.
\end{definition}

From the above discussion, 
every pillow box can be recovered
by fundamental data.

\begin{example}\label{ex:F-data}
We set $b := 1$, and
$\zeta(s) := \sqrt{2} - \sqrt{(s - 1)^2 + 1}$
$(s\in [0,2])$.
Then, the pair $(b, \zeta(s))$
satisfies \eqref{eq:zeta-prime} and \eqref{eq:boundary-condition}.
Hence, $(b, \zeta(s))$ is fundamental data.
Figures~\ref{fig:rectangle-pillow} and~\ref{fig:pillow-box}
exhibit the corresponding pillow box $M$.
\end{example}

Let $g$ be the flat metric on $U = U_+ \cup U_-$ defined by
\begin{equation}\label{eq:1ff-pillow}
  g := ds^2 - 2\, \zeta'(s)\, ds\, dv + dv^2.
\end{equation}
Since
$\| p_s \| = \| p_v \| = \| q_s \| = \| q_v \| = 1$
and 
$p_s \cdot p_v = q_s \cdot q_v = -\zeta'(s)$,
the first fundamental forms $g_p$ and $g_q$ of
the developable surfaces $p(s, v)$ and $q(s, v)$
given in \eqref{eq:pq}
coincide with $g$.

\subsection{Double rectangle}
\label{sec:Developing-maps}

We use the notations in Subsection~\ref{sec:pillow-box-origami}.
Let $\zeta(s)$ $(0 \leq s \leq L)$ be the function given in \eqref{eq:crease}.
We define a curve
$\gamma(s)$ $(0 \leq s \leq L)$
in the $xy$-plane $(z = 0)$ by
\begin{equation}\label{eq:crease-pattern}
  \gamma(s) := \left( \int_{0}^s \sqrt{1 - \zeta'(w)^2}\, dw,\; \zeta(s),\, 0 \right).
\end{equation}
We define $Y : U \to\R^3$ as
\begin{equation}\label{eq:Y}
  Y(s, v) := \gamma(s) + v\, \check{\xi},
\end{equation}
where 
$U$ is a domain as in \eqref{eq:D}, and
$\xi$ is the vector defined in \eqref{eq:xi-xi-hat}.
Then, $Y$ is isometric to $X$.
Namely, since
$\| Y_s \| = \| Y_v \| = 1$
and
$Y_s \cdot Y_v = -\zeta'(s)$,
the first fundamental form $g_Y$ of $Y$ coincides with 
the flat metric $g$ defined in \eqref{eq:1ff-pillow}.
Since $Y$ is a diffeomorphism between $U$ and 
a closed rectangle $\Omega:=Y(U)$ given by
\[
\Omega
= \{ (x, y, 0) \mid 0 \leq x \leq 2a,\; 0 \leq y \leq b \}
\quad
\left( a := \frac{1}{2} \int_{0}^L \sqrt{1 - \zeta'(s)^2}\, ds \right),
\]
the map $Y$ can be regarded as a developing map 
of the quarter domain 
$P$ of the pillow box $M$ given in \eqref{eq:M-entire}.
Since the curve $\gamma(s)$ in $\Omega$
corresponds to the space curve $\vect{c}(s)$,
the plane curve $\gamma(s)$ is 
the crease pattern of $P$ as a curved folding.

\begin{remark}\label{rem:slope1}
By \eqref{eq:zeta-prime} and \eqref{eq:crease-pattern}, 
we observe that
the absolute value of the slope $|dy/dx| = |\zeta'(s)| / \sqrt{1 - \zeta'(s)^2}$ 
of the crease pattern $\gamma(s)$
is less than $1$
(cf.\ the condition~(II)  
in Subsection~\ref{sec:folding-rectangles}).
\end{remark}

The entire pillow box $M$ is recovered 
by taking the union of the reflections as in \eqref{eq:M-entire}.
The entire rectangle can also be restored 
by taking the union of the reflections in the same way.
We denote by $\rho_V$ and $\rho_H$
the reflections with respect to the planes $y = b$ and $z = 0$,
as in \eqref{eq:reflection}.
The union 
\[
R
= \Omega \cup \rho_V(\Omega)
= \{ (x, y, 0) \mid 0 \leq x \leq 2a,\; 0 \leq y \leq 2b \}
\]
is a closed rectangle
with side lengths $2a$ and $2b$.

We set the reflection image $R' := \rho_H(R)$ of $R$ 
with respect to the $xy$-plane.
Then $R'$ coincides with $R$, namely $R'$ is a copy of $R$.

\begin{definition}\label{def:double-rectangle}
Let $\sim$ be the relation on $R \cup R'$
that identifies the corresponding points 
on the four sides of the rectangles $R$ and $R'$.
The quotient space 
\[
\tilde{R} := (R \cup R') / \!\sim
\]
is called a {\it double rectangle}.
\end{definition}
Such a double rectangle $\tilde{R}$ is homeomorphic to the $2$-sphere $S^2$,
and is isometric to the pillow box $M$.

\subsection{Folding rectangles to make pillow boxes}
\label{sec:folding-rectangles}

Here we exhibit a procedure for constructing a pillow box from a double rectangle.

Consider a rectangle $R = \mathrm{ABCD}$ in the Euclidean plane $\R^2$
(see Figure~\ref{fig:R-curve}).
Here, a rectangle means the closed one,
that is, it is the union of its interior and its boundary.
Suppose that the lengths of the sides satisfy $\mathrm{AB} = \mathrm{CD} = 2a$
and $\mathrm{BC} = \mathrm{DA} = 2b$,
where $a, b$ are positive numbers.
Let $R' = \mathrm{ABCD}$ be a copy of $R$.

\begin{figure}[htb]
\centering
 \begin{tabular}{c@{\hspace{7mm}}c}
  \resizebox{4cm}{!}{\includegraphics{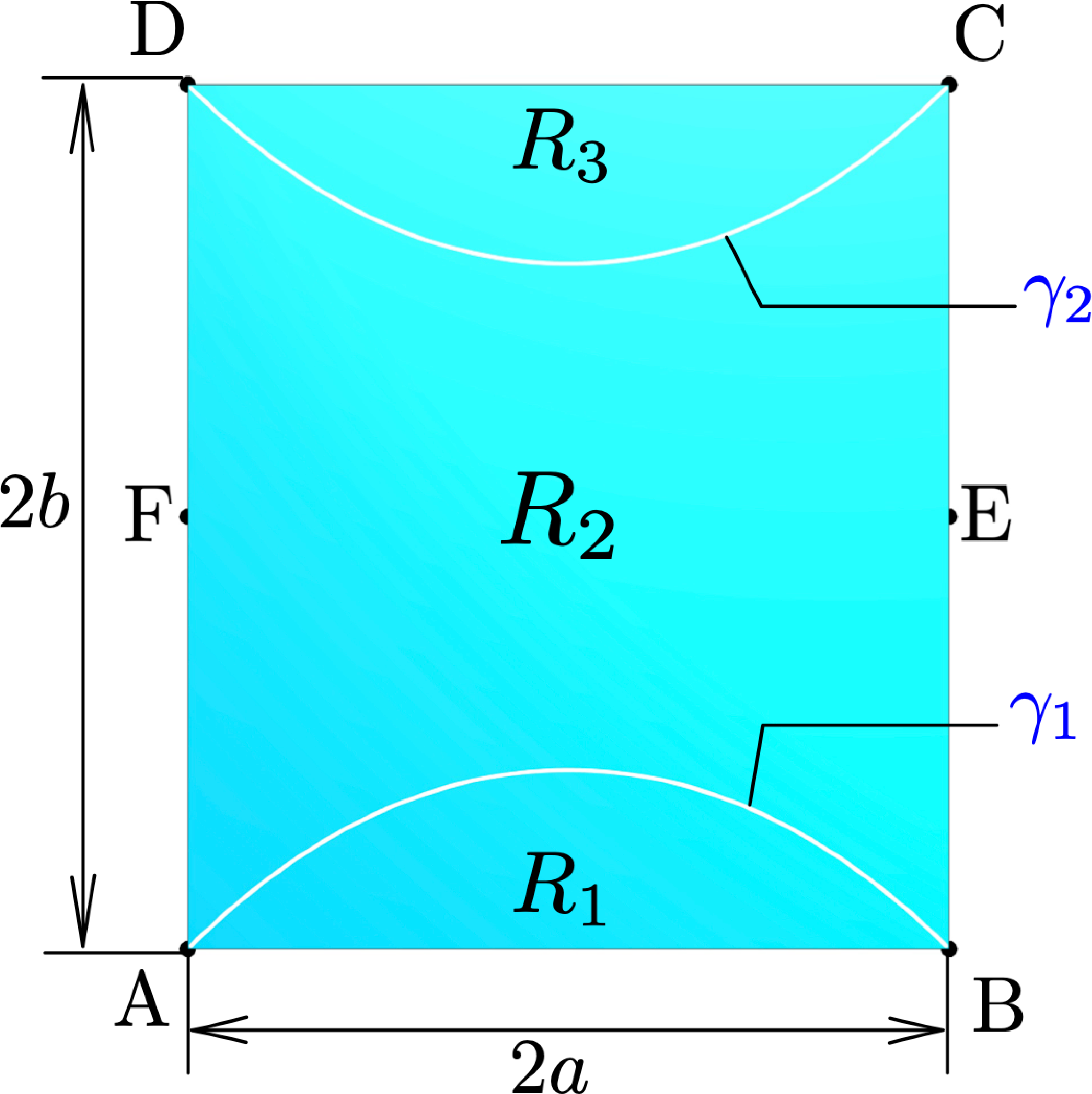}} &
  \resizebox{4cm}{!}{\includegraphics{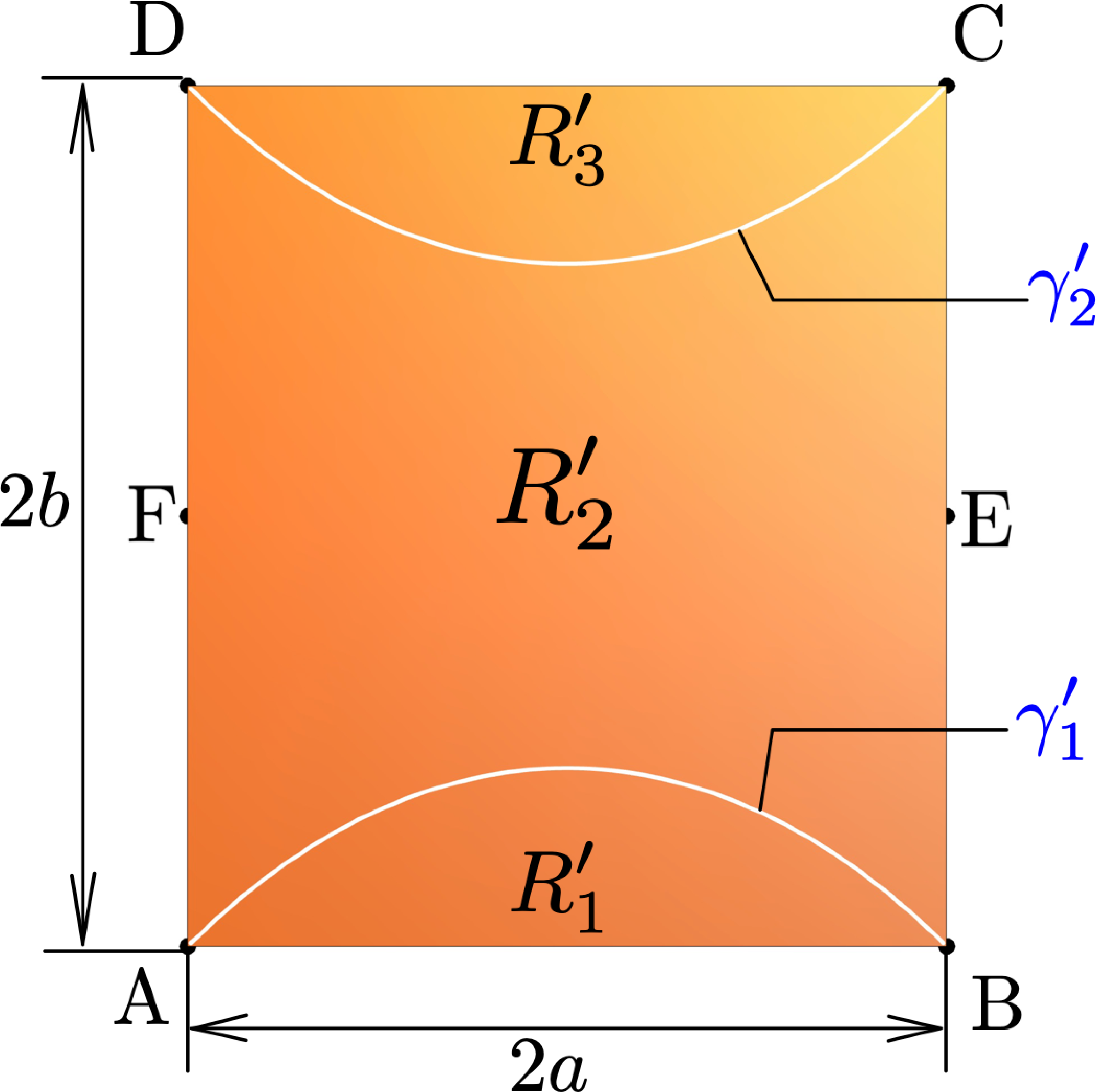}}
 \end{tabular}
 \caption{Regular curves $\gamma_1, \gamma_2$ in 
the rectangle $R = R_1 \cup R_2 \cup R_3$
and $\gamma_1', \gamma_2'$ in 
the other rectangle $R' = R_1' \cup R_2' \cup R_3'$.}
\label{fig:R-curve}
\end{figure}

Let $\mathrm{E}$ and $\mathrm{F}$ be the midpoints of
$\mathrm{BC}$ and $\mathrm{DA}$, respectively. 
Consider smooth curves $\gamma_1$ and $\gamma_2$ in $R$
(see Figure~\ref{fig:R-curve})
which satisfy the following four conditions:
\begin{itemize}
\item[(I)] The image of the curve $\gamma_1$ 
is given by a graph over the segment $\mathrm{AB}$,
which includes $\mathrm{A}$ and $\mathrm{B}$ as
the endpoints of $\gamma_1$.
\item[(II)] 
The absolute value of the slope
of the image of $\gamma_1$
is less than $1$ (cf.\ Remark~\ref{rem:slope1}).
\item[(III)] 
The curvature of $\gamma_1$ does not vanish.
\item[(IV)]
The image of $\gamma_1$ is contained in 
the rectangle $\mathrm{ABEF}$,
$\gamma_1$ and $\gamma_2$ do not intersect, 
and 
$\gamma_2$ coincides with the reflection of $\gamma_1$
with respect to the line $\mathrm{EF}$.
\end{itemize}

Let $\gamma_1'$ and $\gamma_2'$
be smooth curves in $R'$ 
which are copies of $\gamma_1$ and $\gamma_2$, respectively.
We divide $R$ (resp.\ $R'$) 
into the three closed domains $R_1$, $R_2$, and $R_3$
(resp.\ $R_1'$, $R_2'$, and $R_3'$)
by $\gamma_1$ and $\gamma_2$ 
(resp.\ $\gamma_1'$ and $\gamma_2'$)
as in Figure~\ref{fig:R-curve}.

Then, the constructive definition of pillow boxes
is stated as follows.
(See Definition~\ref{def:pillow-box} for the formal definition.)
Regarding $R$ and $R'$ as sheets of paper,
we fold the rectangle $R$ into a mountain shape $S$ along 
$\gamma_1$ and $\gamma_2$,
and fold $R'$ into a valley shape $S'$ along 
$\gamma_1'$ and $\gamma_2'$.
Connecting them along the boundaries, we obtain a box
$M = S \cup S'$
(see Figure~\ref{fig:pillow-box}).
For $i = 1, 2, 3$, let $S_i$ and $S_i'$ 
be the closed domains 
corresponding to $R_i$ and $R_i'$,
respectively.
If $S_1 \cup S_1'$, $S_3 \cup S_3'$, $S_2$, and $S_2'$
are all subsets of cylinders,
then $M$ is a pillow box.
Here, a {\it cylinder} 
means a connected regular surface in $\R^3$ 
given by the product of a regular plane curve and 
an entire straight line.

\begin{figure}[htb]
\centering
 \begin{tabular}{c}
  \resizebox{5cm}{!}{\includegraphics{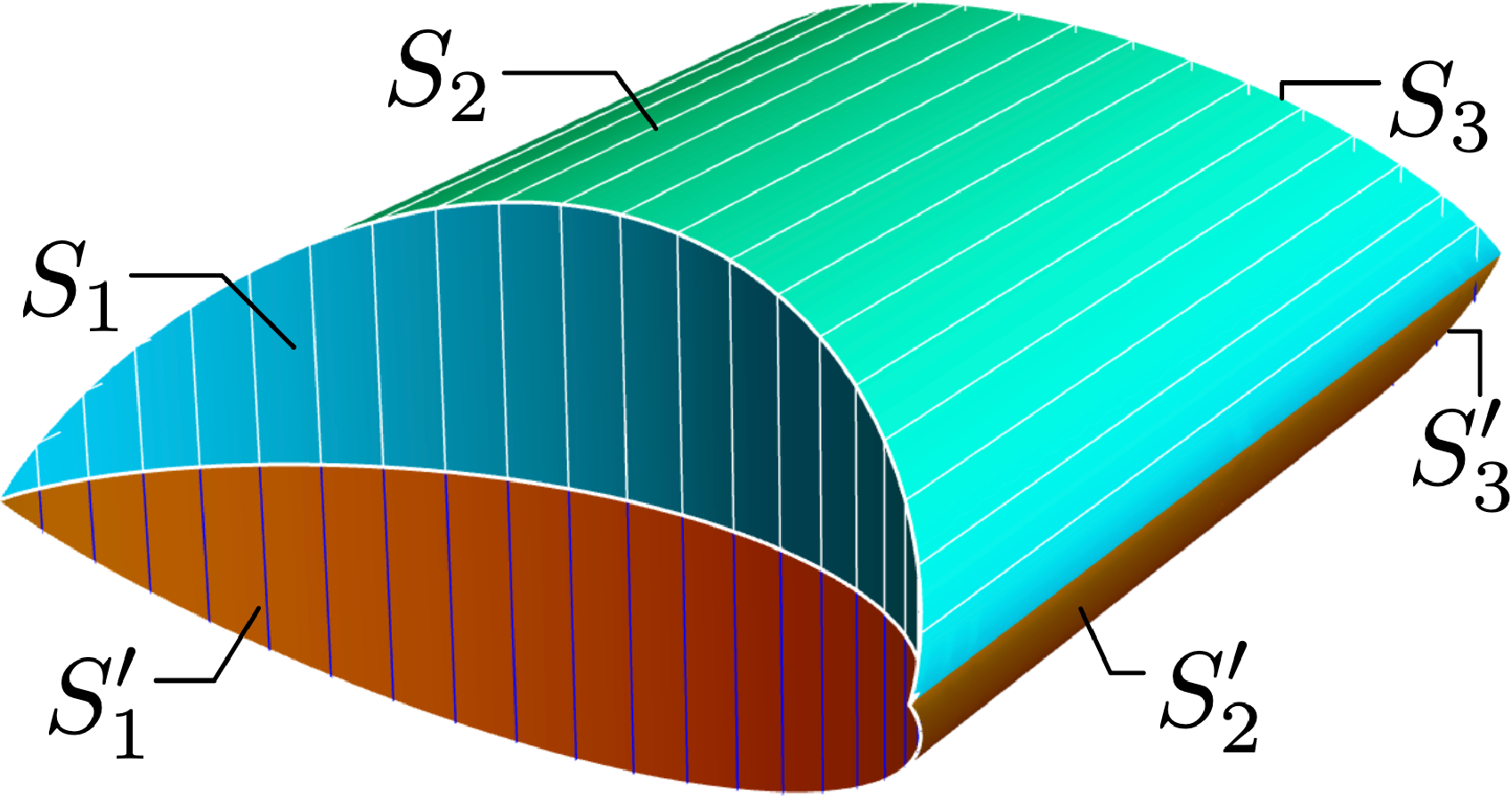}}
 \end{tabular}
 \caption{A pillow box 
 is a union of subsets of cylinders
 foliated by line segments.}
\label{fig:pillow-box}
\end{figure}

The explicit construction method for creating 
a pillow box from a double rectangle is given as follows:
By parallel translation, we may assume that 
\[
\mathrm{A} = (0, 0), \quad
\mathrm{B} = (2a, 0), \quad
\mathrm{C} = (2a, 2b), \quad
\mathrm{D} = (0, 2b).
\]
We denote by $\gamma_1(s) = (\varphi_1(s),\, \zeta_1(s))$ $(s \in [0, L])$
the arc-length parametrization of the curve $\gamma_1$ 
in the rectangle $R = \mathrm{ABCD}$ 
that satisfies conditions (I)--(IV) above,
where $L$ is the length of $\gamma_1$.
By conditions (I)--(IV), we have
$\zeta_1(0) = \zeta_1(L) = 0$
and
\[
1 - 2\, \zeta_1'(s)^2 > 0,\quad
\zeta_1''(s) < 0,\quad
0 < \zeta_1(s) < b
\quad (0 < s < L).
\]
Thus, the pair $(b, \zeta_1(s))$ satisfies 
\eqref{eq:zeta-prime} and \eqref{eq:boundary-condition}.
Namely, $(b, \zeta_1(s))$ is the fundamental data of 
a pillow box $M$ (cf.\ Definition~\ref{def:F-data}).

\section{Origami deformations of pillow boxes}
\label{sec:origami-deformation}

In the previous section, 
we observed that every pillow box is isometric to a double rectangle
given by its developing image.
Therefore, it is natural to consider the existence of isometric deformations
between them.
One example of such an isometric deformation is given as follows:  
Let $M$ be a pillow box defined by 
the fundamental data $(b, \zeta(s))$.
The corresponding crease pattern $\gamma(s)$
is given by \eqref{eq:crease-pattern}.
We write its graph form as 
$\gamma(x) = (x, \psi(x))$ for $0 \leq x \leq 2a$.  
For each $t \in [0, 1)$, define
\[
\gamma^t(x) := \big(x,\, (1 - t)\psi(x)\big)
\quad (0 \leq x \leq 2a).
\]
Since $\gamma(x)$ satisfies conditions (I)--(IV)  
in Subsection~\ref{sec:folding-rectangles},  
so does $\gamma^t(x)$.  
And hence, $\gamma^t(x)$ is a crease pattern
of a pillow box $\bar{M}^t$
as described in Subsection~\ref{sec:folding-rectangles}.
By construction, the family $\{\bar{M}^t\}$  
shares a common double rectangle $\tilde{R}$  
and thus defines an isometric deformation  
from $M = \bar{M}^0$ to $\tilde{R} = \lim_{t \to 1} \bar{M}^t$  
(see Figure~\ref{fig:pillow-0-deformation-long}).

\begin{figure}[htb]
\begin{center}
 \begin{tabular}{c@{\hspace{4mm}}c@{\hspace{4mm}}c}
  \resizebox{4cm}{!}{\includegraphics{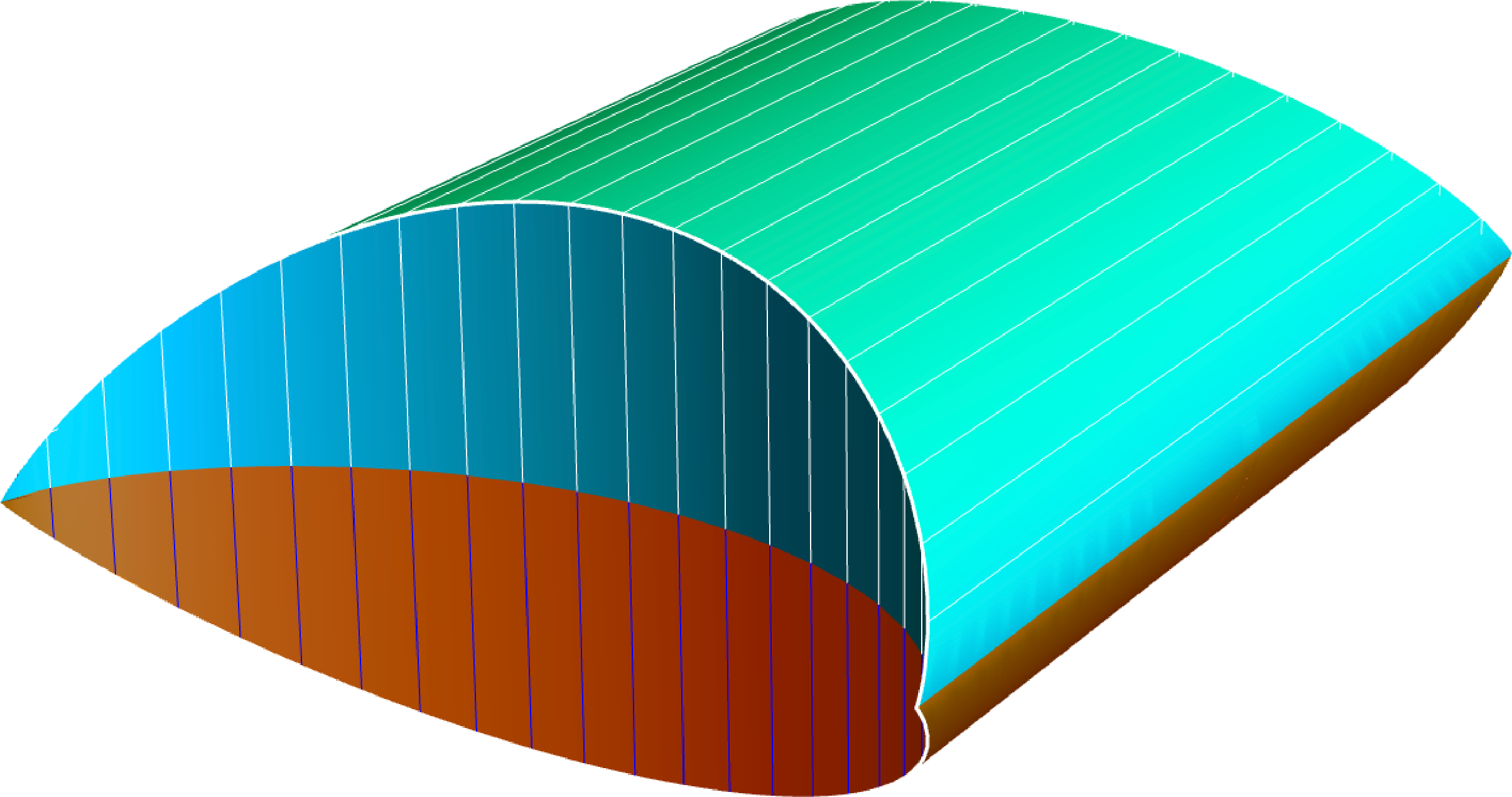}}&
  \resizebox{4cm}{!}{\includegraphics{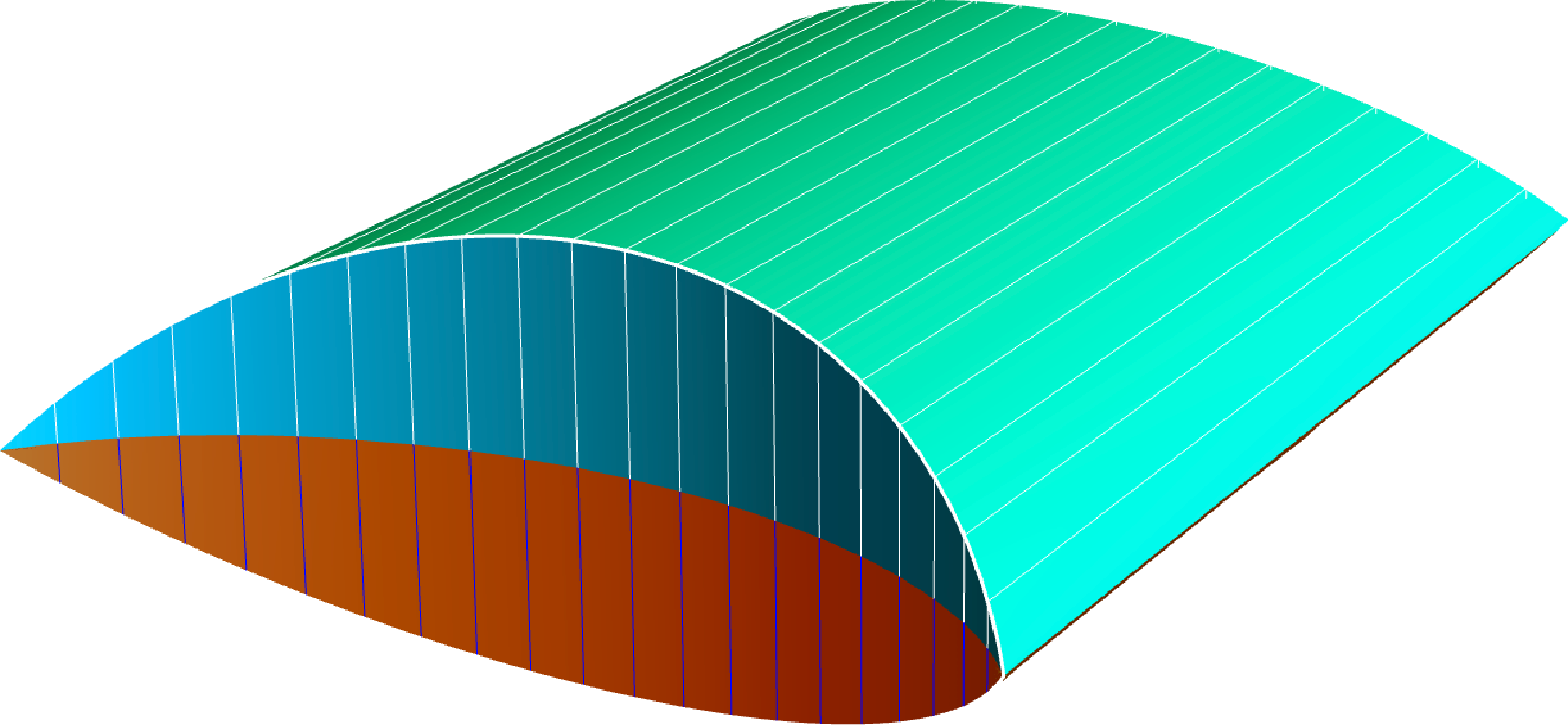}}&
  \resizebox{4cm}{!}{\includegraphics{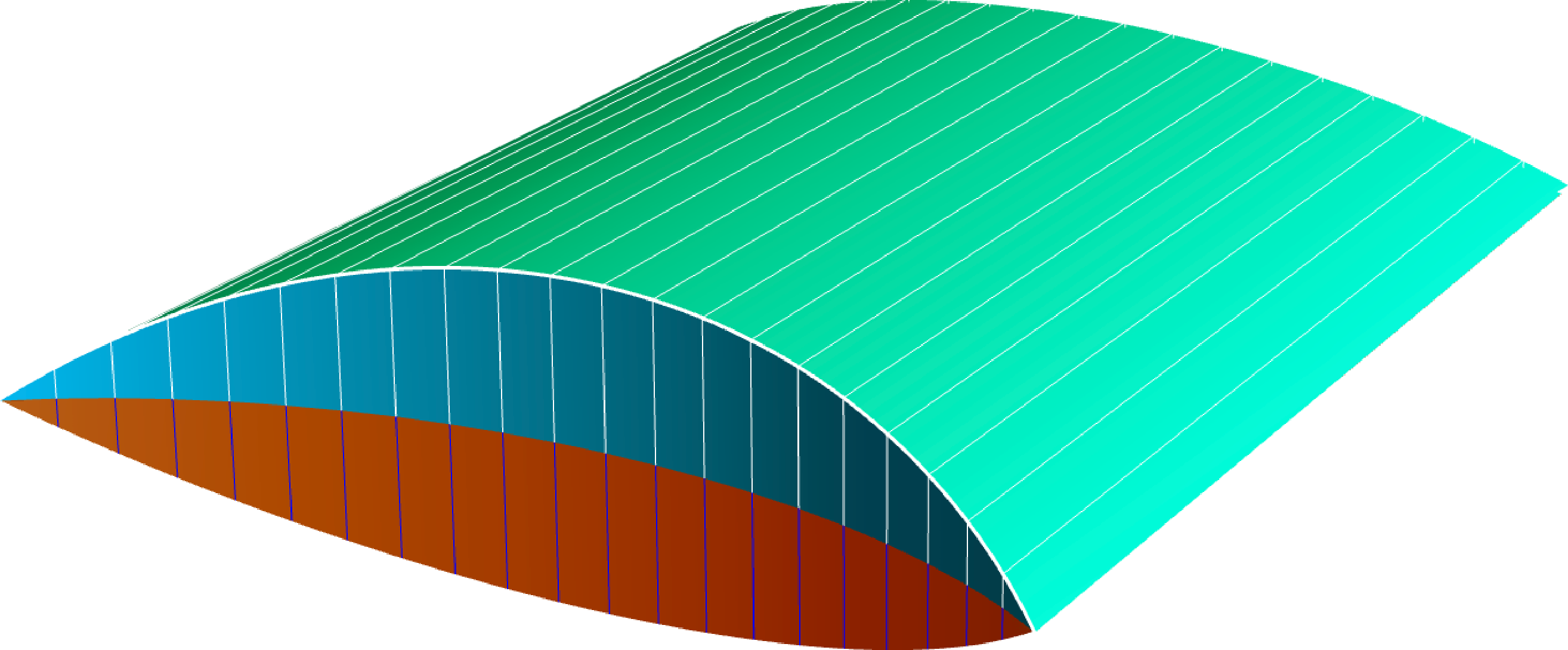}}\\
\vspace{0.5cm}
  {\footnotesize $M = \bar{M}^0$}&
  {\footnotesize $\bar{M}^{0.2}$}&
  {\footnotesize $\bar{M}^{0.4}$}\\
  \resizebox{4cm}{!}{\includegraphics{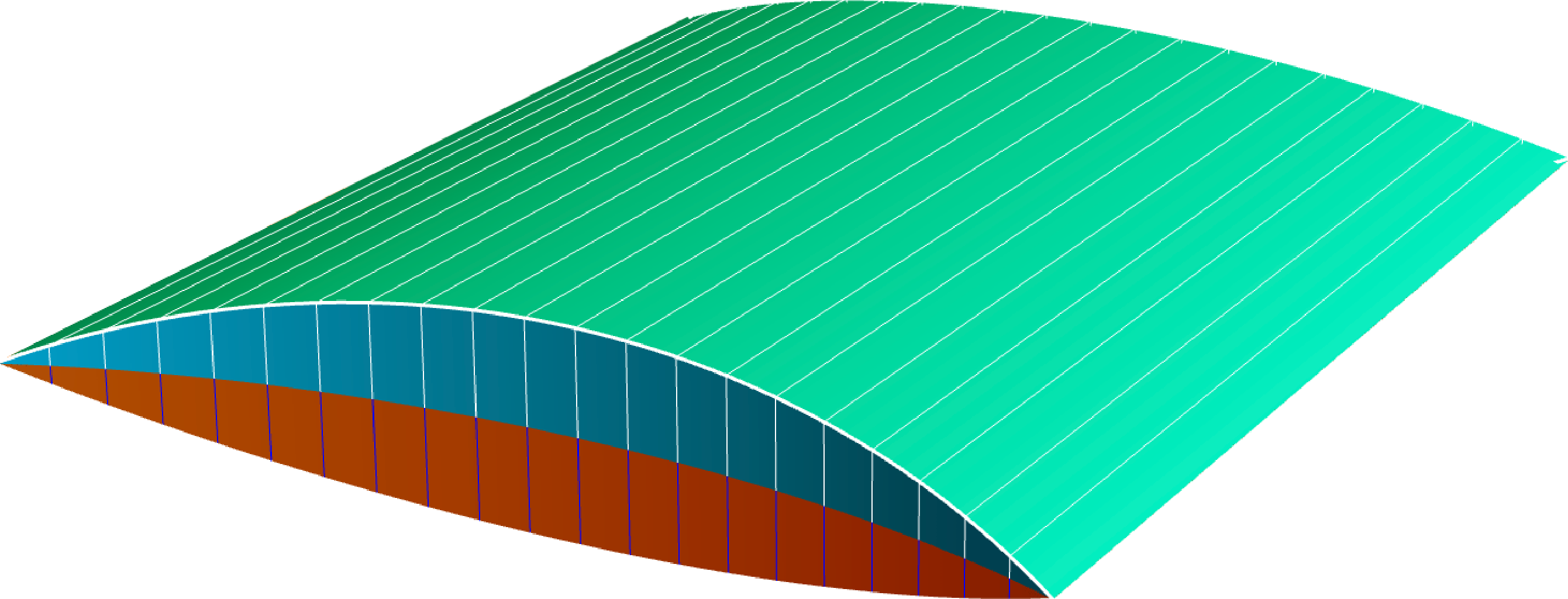}}&
  \resizebox{4cm}{!}{\includegraphics{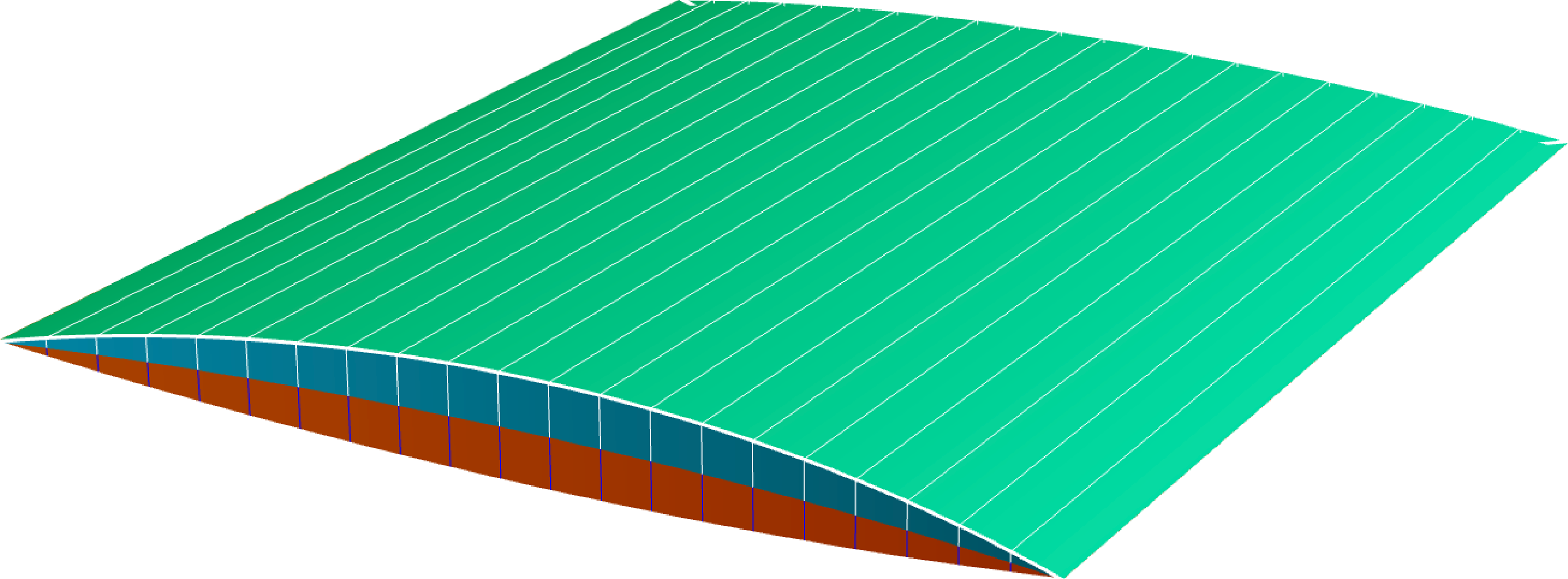}}&
  \resizebox{4cm}{!}{\includegraphics{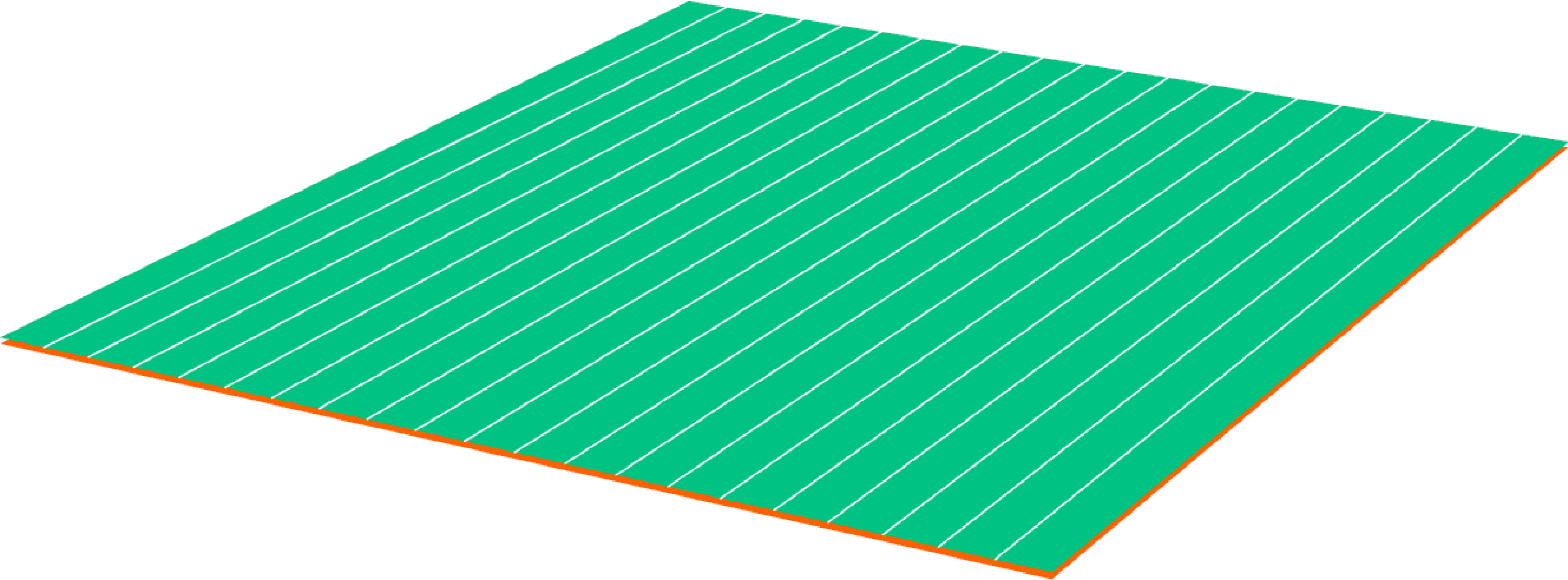}}\\
  {\footnotesize $\bar{M}^{0.6}$}&
  {\footnotesize $\bar{M}^{0.8}$}&
  {\footnotesize $\tilde{R} = \lim_{t \to 1} \bar{M}^t$}
 \end{tabular}
\end{center}
 \caption{The family $\{ \bar{M}^t \}_{t \in [0, 1]}$ of 
 pillow boxes that 
 gives an isometric deformation from $M = \bar{M}^0$
 to the corresponding double rectangle $\tilde{R} = \lim_{t \to 1} \bar{M}^t$.
 Note that the family $\{ \bar{M}^t \}_{t \in [0, 1]}$ changes the crease pattern.}
\label{fig:pillow-0-deformation-long}
\end{figure}

This deformation changes the crease pattern $\gamma$.
However, in practical applications, 
isometric deformations that preserve the crease pattern are desirable.
In this section, we define {\it origami deformations},
which are isometric deformations from pillow boxes to double rectangles 
among curved foldings that preserve the crease pattern
(Definition~\ref{def:origami-deformation}).
In Theorem~\ref{thm:Koiso-deformation}, we derive a representation formula 
for quarter origami deformations, 
which yields an explicit classification of origami deformations of pillow boxes.
As a corollary, 
we prove that such an origami deformation 
necessarily changes the topology of a pillow box
(Corollary~\ref{cor:isometric-deformation}).

\subsection{Origami deformations}

To define an origami deformation (Definition~\ref{def:origami-deformation}), 
we first define a {\it quarter origami deformation}
(Definition~\ref{def:Q-origami-deformation}),
which is an isometric deformation from a quarter domain of a pillow box 
to a rectangle among curved foldings 
that preserve the crease pattern.
In the following, 
we use the notations in Subsections~\ref{sec:pillow-box-origami} 
and~\ref{sec:Developing-maps}.

\begin{definition}\label{def:Q-origami-deformation}
Let $M$ be a pillow box
given by fundamental data $(b, \zeta(s), 0\leq s\leq L)$
(cf.\ Definition~\ref{def:F-data}).
Consider a continuous 
$1$-parameter family $\{ X^t : U \to\R^3 \}_{t \in [0,1]}$
of origami maps,
described as
\begin{equation}\label{eq:Xt}
X^t(s, v) =
\begin{cases}
p^t(s, v) = \vect{c}^t(s) + v\, \xi^t(s) & ((s, v) \in U_+), \\[6pt]
q^t(s, v) = \vect{c}^t(s) + v\, \check{\xi} & ((s, v) \in U_-), 
\end{cases}
\end{equation}
for each $t \in [0,1]$,
where
$\check{\xi}$, $U_+$ and $U_-$
are defined by 
\eqref{eq:xi-xi-hat}, 
\eqref{eq:D+}, 
\eqref{eq:D-},
respectively.
Suppose that the following hold:
\begin{itemize}
\item[(1)]
$X^0 = X$ and $X^1 = Y$ hold,
where $X$ and $Y$ are defined by 
\eqref{eq:pillow-origami}
and~\eqref{eq:Y},
respectively.
\item[(2)]
For each $t \in [0,1]$,
$X^t$ is isometric to both $X$ and $Y$.
Namely, 
the first fundamental forms $g_{p^t}$ and $g_{q^t}$ of  
$p^t$ and $q^t$ coincide with $g$ on $U_+$ and $U_-$,
respectively,
where $g$ is the common first fundamental form of $X$ and $Y$
given in \eqref{eq:1ff-pillow}.
\item[(3)] 
The endpoints $\vect{c}^t(0)$ and $\vect{c}^t(L)$
of the crease of the origami map $X^t$
lie on the $x$-axis for each $t \in [0,1]$.
\end{itemize}
Then $\{ X^t : U \to\R^3 \}_{t \in [0,1]}$
is called a {\it quarter origami deformation}
from $X$ to $Y$.
Furthermore, we refer to $\Phi^t(s) := X^t(s, \zeta(s)-b)$ as the {\it vertical end}, 
and $\Psi^t(s) := X^t(s, \zeta(s))$ as the {\it horizontal end}.
\end{definition}

Condition (3) is necessary to make the pillow box closed.
The following proposition implies that 
every quarter origami deformation can be reflected 
to obtain the entire image of the origami maps.

\begin{proposition}\label{prop:upper-end}
Let $\{ X^t : U \to\R^3 \}_{t \in [0,1]}$ be a
quarter origami deformation
from $X$ to $Y$.
Then, for each $t \in [0,1]$, 
the vertical end $\Phi^t(s)$ lies in the vertical plane $y = b$.
In particular, the curved folding
$P^t := X^t(U)$ and its reflection image $\rho_{V}(P^t)$ are connected smoothly,
where $\rho_V$ is the reflection with respect to 
the vertical plane $V_M$ given in \eqref{eq:reflection}.
\end{proposition}

To prove this proposition,
we prepare the following lemma.

\begin{lemma}\label{lem:y}
Let $\{ X^t : U \to\R^3 \}_{t \in [0,1]}$ be a
quarter origami deformation
from $X$ to $Y$.
Write $\vect{c}^t(s) = (x^t(s),\, y^t(s),\, z^t(s))$.
Then,
$\vect{c}^t : [0, L] \to \R^3$ is parametrized by arc length,
and $y^t(s) = \zeta(s)$ holds for each $t \in [0,1]$.
\end{lemma}

\begin{proof}
The first fundamental form $g_{q^t}$ of $q^t(s, v)$ is written as
\[
g_{q^t} = \| (\vect{c}^t)'(s) \|^2\, ds^2 - 2\, (y^t)'(s)\, ds\, dv + dv^2.
\]
By (2) of Definition~\ref{def:Q-origami-deformation},
we have 
$\| (\vect{c}^t)'(s) \| = 1$
and 
$(y^t)'(s) = \zeta'(s)$.
The first equation yields the former assertion.
By the second equation, there exists a function $b(t)$ such that 
$y^t(s) = \zeta(s) + b(t)$.
Condition (3) implies that $y^t(L) = 0$ for each $t \in [0,1]$.
Since $\zeta(L) = 0$, the function $b(t)$ must be identically zero.
Hence, the latter assertion follows.
\end{proof}

\begin{proof}[Proof of Proposition~\ref{prop:upper-end}]
We fix $t\in [0,1]$ arbitrarily.
By Lemma~\ref{lem:y},
\[
\Phi^t(s)
=
X^t(s, \zeta(s) - b)
=
q^t(s, \zeta(s) - b)
=
(x^t(s),\, b,\, z^t(s))
\]
holds.
Hence, the vertical end $\Phi^t(s)$
is included in the vertical plane $\{y=b\}$.
Moreover,
since both $X^t(U_-)$ and $\rho_{V}\circ X^t(U_-)$ 
are subsets of a single cylinder 
$\{(x^t(s),y,z^t(s)) \in \R^3 \mid s\in [0,L],\,y\in \R\}$,
we obtain the assertion.
\end{proof}

Let $\{ X^t : U \to\R^3 \}_{t \in [0,1]}$ 
be a quarter origami deformation 
as in Definition~\ref{def:Q-origami-deformation}.
By condition (3) in Definition~\ref{def:Q-origami-deformation},
the curved folding $P^t := X^t(U)$ 
and its reflection image $\rho_{H}(P^t)$ are connected
continuously for any $t \in [0,1]$,
where $\rho_H$ is the reflection with respect to 
the horizontal plane $V_H$ given in \eqref{eq:reflection}.
Thus, we obtain the following definition:

\begin{definition}\label{def:origami-deformation}
We set $M^t$ as the union
\[
M^t := P^t \cup \rho_{V}(P^t) \cup \rho_{H}(P^t) \cup \rho_{H} \circ \rho_{V}(P^t).
\]
We call $\{ M^t \}_{t \in [0,1]}$
an {\it origami deformation} from 
a pillow box $M$ to a double rectangle $\tilde{R}$.
\end{definition}

\subsection{Classification of origami deformations}

We derive the following representation formula 
for quarter origami deformations,
which yields an explicit classification of origami deformations.

\begin{theorem}\label{thm:Koiso-deformation}
Let $M$ be a pillow box
given by fundamental data $(b, \zeta(s), 0\leq s\leq L)$.
Denote by $X : U \to\R^3$ an origami parametrization 
of a quarter domain $P$ of $M$,
with developing map $Y : U \to\R^3$.
Suppose there exist continuous functions 
$\lambda(t)$ and $\mu(t)$ $(t \in [0,1])$
such that
$$
\mu(0) = \mu(1) = \lambda(1) = 0,\quad
\lambda(0) = 1,\quad
( 1 + \lambda(t)^2 )\zeta'(s)^2 < 1
$$
hold for every $t\in [0,1]$ and $s\in (0,L)$.
Then, define $X^t : U \to\R^3$ for each $t \in [0,1]$ by
\[
 X^t(s, v) :=
 \begin{cases} 
 p^t(s, v) = \vect{c}^t(s) + v\,\xi^t & ((s,v) \in U_+),\\
 q^t(s, v) = \vect{c}^t(s) + v\,\check{\xi} & ((s,v) \in U_-),
 \end{cases}
\]
where $\check{\xi} = (0,-1,0)$, and
\begin{align}
\label{eq:crease-t}
\vect{c}^t(s) 
&= 
\left(
\int_{0}^{s} \sqrt{1 - (1 + \lambda(t)^2) \zeta'(w)^2} \, dw + \mu(t),\;
\zeta(s),\; \lambda(t)\, \zeta(s)
\right),\\
\label{eq:xi-check-t}
\xi^t
&= \frac{1}{1 + \lambda(t)^2} \left( 0,\; \lambda(t)^2-1,\; -2\lambda(t) \right).
\end{align}
Then, the family $\{ X^t : U \to\R^3 \}_{t \in [0,1]}$ is a quarter origami deformation
from $X$ to $Y$.
Conversely, 
every quarter origami deformation
$\{ X^t : U \to\R^3 \}_{t \in [0,1]}$
is given in this form.
\end{theorem}

For the proof, 
we prepare the following lemmas
(Lemmas~\ref{lem:1st-FF-pq} and~\ref{lem:yz-plane}).

\begin{lemma}\label{lem:1st-FF-pq}
Let $p, q : U \to\R^3$ be mutually dual developable surfaces 
along 
a space curve $\vect{c}(s)$ with positive curvature
$($cf.\ Definition~$\ref{def:curved-folding})$.
Then, the following are equivalent:
\begin{itemize}
\item[{\rm (i)}]
$\vect{c}(s)$ is a plane curve.
\item[{\rm (ii)}]
The second angular function
$\beta(s)$ of $p(s,v)$
coincides with that $\check{\beta}(s)$
of $q(s,v)$.
\item[{\rm (iii)}]
The first fundamental form $g_p$
of $p(s,v)$
coincides with that $g_q$
of $q(s,v)$.
\end{itemize}
\end{lemma}

We note that
Lemma~\ref{lem:1st-FF-pq} is a well-known fact 
in the field of curved folding.
For the sake of self-containment, we provide a proof here.

\begin{proof}[Proof of Lemma~\ref{lem:1st-FF-pq}]
Let $\alpha(s)$ be the first angular function of $p(s,v)$.  
Since $q(s,v)$ is the dual of $p(s,v)$,
the first angular function of $q(s,v)$ is $\check{\alpha}(s) = -\alpha(s)$.
By \eqref{eq:beta},
\[
\cot \check{\beta}(s) = \frac{\alpha'(s) - \tau(s)}{\kappa(s)\sin \alpha(s)}.
\]
In particular, $\beta(s) = \check{\beta}(s)$
if and only if the torsion $\tau(s)$ vanishes identically,
i.e., $\vect{c}(s)$ is planar,
thereby yielding the equivalence of (i) and (ii).

With respect to (ii) and (iii),
note that
the geodesic curvature $\kappa_g(s)$ of $\vect{c}(s)$ 
as a curve on $p(s,v)$
coincides with that as a curve on $q(s,v)$
by \eqref{eq:kappa-g}.
The first fundamental forms $g_p$ and $g_q$ are given by
(cf.\ \cite[Equation (1.7)]{HNSUY3})
\begin{align}
\label{eq:1st-FF-p}
g_p &= 
\Bigl( \bigl( \sin \beta - v(\beta' + \kappa_g) \bigr)^2 + \cos^2 \beta \Bigr)\, ds^2
+ 2 \cos \beta\, ds\, dv + dv^2,\\
\label{eq:1st-FF-q}
g_q &= 
\Bigl( \bigl( \sin \check{\beta} - v(\check{\beta}' + \kappa_g) \bigr)^2 + \cos^2 \check{\beta} \Bigr)\, ds^2
+ 2 \cos \check{\beta}\, ds\, dv + dv^2.
\end{align}
As $\kappa_g\ne0$ by \eqref{eq:kappa-g},
$g_p = g_q$ if and only if $\beta = \check{\beta}$,
proving the equivalence.
\end{proof}

\begin{lemma}\label{lem:yz-plane}
Let $\{ X^t : U \to\R^3 \}_{t \in [0,1]}$ be a
quarter origami deformation
from $X$ to $Y$
$($cf.\ Definition~$\ref{def:Q-origami-deformation})$. Then, there
exists a continuous function $\lambda(t)$ $(t \in [0,1])$
such that 
$\lambda(0) = 1$,
$\lambda(1) = 0$, and
the crease $\vect{c}^t(s)$ lies on the plane 
$z = \lambda(t)\, y$ in $\R^3$
for each $t \in [0,1]$.
\end{lemma}

\begin{proof}
By (2) of Definition~\ref{def:Q-origami-deformation}
and Lemma~\ref{lem:1st-FF-pq},
the crease $\vect{c}^t(s)$ is a planar curve.
Hence, there exist functions 
$a(t)$, $b(t)$, $c(t)$, and $d(t)$
such that
$\vect{c}^t(s)$ lies on the plane
\[
\Pi^t := \{ (x,y,z) \in \R^3 \mid a(t)\, x + b(t)\, y + c(t)\, z = d(t) \}
\quad (t \in [0,1]).
\]
By (3) of Definition~\ref{def:Q-origami-deformation},
the endpoints of $\vect{c}^t(s)$
are written as
$\vect{c}^t(0) = (\mu(t), 0, 0)$
and 
$\vect{c}^t(L) = (\eta(t), 0, 0)$,
where $\mu(t)$ and $\eta(t)$ are continuous functions on $[0,1]$
satisfying $\mu(t) < \eta(t)$.
By (1) of Definition~\ref{def:Q-origami-deformation},
we have $\mu(0) = \mu(1) = 0$.
Since the crease $\vect{c}^t(s)$ 
lies on $\Pi^t$ for each $t$,
we have $a(t)\, \mu(t) = a(t)\, \eta(t) = d(t)$.
Hence, $a(t)$ and $d(t)$ must be identically zero.

Next, we claim that $c(t)$ is nonzero for each $t$.
If not, there exists $t_0 \in [0,1]$ such that $c(t_0) = 0$.
By Lemma~\ref{lem:y}, the second component of 
$\vect{c}^t(s)$ is $\zeta(s)$, so $b(t_0)\, \zeta(s) = 0$.
Then, by \eqref{eq:boundary-condition}, we have $b(t_0) = 0$,
which contradicts the assumption that $\Pi^t$ is a plane.
Therefore, $c(t) \neq 0$ for all $t$.
Setting $\lambda(t) := -b(t)/c(t)$,
the crease $\vect{c}^t(s)$ lies on the plane $z = \lambda(t)\, y$.
By (1) of Definition~\ref{def:Q-origami-deformation},
we have $\lambda(0) = 1$, $\lambda(1) = 0$,
proving the assertion.
\end{proof}

\begin{proof}[Proof of Theorem~\ref{thm:Koiso-deformation}]
By direct calculation, the first assertion follows.  
We now prove the converse:
Let $\vect{c}^t : [0, L] \to \R^3$ 
be the crease of the origami map $X^t$,
and write $\vect{c}^t(s) = (x^t(s),\, y^t(s),\, z^t(s))$.
By Lemma~\ref{lem:y}, we have $y^t(s) = \zeta(s)$.
Then, by Lemma~\ref{lem:yz-plane},
there exists a continuous function $\lambda(t)$ $(t \in [0,1])$
such that $z^t(s) = \lambda(t)\, \zeta(s)$.
Since $\vect{c}^t(s) = (x^t(s), \zeta(s), \lambda(t)\zeta(s))$ 
is parametrized by arc length 
(Lemma~\ref{lem:y}), we have
\begin{equation}\label{eq:x^t-prime}
|(x^t)'(s)| = \sigma^t(s),
\quad
\text{where} \quad
\sigma^t(s) := \sqrt{1 - (1 + \lambda(t)^2) \zeta'(s)^2}.
\end{equation}
Now, we prove by contradiction that 
$\mathcal{Z} := \{(t,s) \in [0,1] \times (0,L) \mid \sigma^t(s) = 0\}$ 
is empty.
Suppose, for contradiction, that $\mathcal{Z}$ is non-empty.
Since $\lambda(0) = 1$ and by \eqref{eq:zeta-prime}, 
the function $\sigma^t(s)$ is not identically zero on $[0,1] \times (0,L)$.
Thus, there exist a point $(t_0, s_0) \in \mathcal{Z}$ and a sequence 
$\{(t_n, s_n)\}_{n=1}^\infty$ in $[0,1] \times (0,L) \setminus \mathcal{Z}$
such that $(t_n, s_n) \to (t_0, s_0)$ as $n \to \infty$.
By \eqref{eq:x^t-prime}, the curvature function
$\kappa^t(s) := \| (\vect{c}^t)''(s) \|$
is given by
\[
\kappa^t(s) =
-\, \frac{ \psi(t)\, \zeta''(s) }{ \sigma^t(s) },
\quad \text{where} \quad
\psi(t) := \sqrt{1 + \lambda(t)^2 }.
\]
By \eqref{eq:boundary-condition}, 
it follows that
$\kappa^{t_n}(s_n)$ diverges to $+\infty$ as $n \to \infty$,
which is a contradiction.
Hence, 
$\mathcal{Z} = \varnothing$,
i.e.,
$\sigma^t(s)$ is positive for all $t \in [0,1]$ and $s \in (0,L)$.
By \eqref{eq:x^t-prime}, we have $(x^t)'(s)=\sigma^t(s)$.
Setting $\mu(t) := x^t(0)$ gives \eqref{eq:crease-t}.

The principal normal vector field $N^t(s)$
and the binormal vector field $B^t(s)$
are given by
\[
N^t(s) = 
\left( \psi(t)\, \zeta'(s),\, 
-\frac{\sigma^t(s)}{\psi(t)},\, 
-\frac{\lambda(t)\, \sigma^t(s) }{\psi(t)}\right),
\quad
B^t(s) = 
\left(0,\, \frac{\lambda(t)}{\psi(t)},\, -\frac{1}{\psi(t)}\right),
\]
respectively.
The unit conormal vector field
$\check{\vect{n}}_g^t(s)$ on $q^t(s,v)$ along $\vect{c}^t(s)$ is
\[
\check{\vect{n}}_g^t(s) =
\frac{1}{ \sqrt{1 - \zeta'(s)^2} }
\Bigl( \zeta'(s)\, \sigma^t(s),\, -1 + \zeta'(s)^2,\, \lambda(t)\, \zeta'(s)^2 \Bigr).
\]
Thus, the first angular function $\alpha^t(s)$
of $q^t(s,v)$ satisfies
\begin{equation}\label{eq:alpha-t}
  \cos \alpha^t(s)
  = \frac{ \sigma^t(s) }{ \psi(t) \sqrt{1 - \zeta'(s)^2} },
  \quad
  \sin \alpha^t(s)
  = -\, \frac{ \lambda(t) }{ \psi(t) \sqrt{1 - \zeta'(s)^2} }.
\end{equation}

By (2) of Definition~\ref{def:Q-origami-deformation},
and from \eqref{eq:1st-FF-p} and \eqref{eq:1st-FF-q},
the second angular functions of $p^t(s,v)$ and $q^t(s,v)$ 
coincide with $\beta(s)$ for $p(s,v)$ given by \eqref{eq:beta-p}.
Since $X^t : U \to\R^3$ is an origami map,
$X^t(s,v) = p^t(s,v)$ holds on $U_+$,
where $p^t(s,v) = \vect{c}^t(s) + v\, \xi^t(s)$
and 
\[
\xi^t(s) = \cos \beta(s)\, T^t(s) 
+ \sin \beta(s)\, \bigl( \cos \alpha^t(s)\, N^t(s) - \sin \alpha^t(s)\, B^t(s) \bigr).
\]
Substituting \eqref{eq:alpha-t} and \eqref{eq:beta-p}
yields \eqref{eq:xi-check-t},
which completes the proof.
\end{proof}

Using the formula in Theorem~\ref{thm:Koiso-deformation},
we can construct explicit examples of origami deformations for pillow boxes.

\begin{example}\label{ex:Koiso-deformation}
Let $(b, \zeta(s))$ be the fundamental data given in 
Example~\ref{ex:F-data}.
Then, $\mu(t) := 0$ and $\lambda(t) := 1 - t$
for $t \in [0,1]$
are continuous functions satisfying the conditions 
of Theorem~\ref{thm:Koiso-deformation}.
Figures~\ref{fig:pillow-deformation} 
and~\ref{fig:pillow-deformation-long}
illustrate the resulting origami deformation 
$\{ M^t \}_{t \in [0,1]}$.
\end{example}

\begin{figure}[htb]
\centering
\begin{tabular}{c@{\hspace{4mm}}c@{\hspace{4mm}}c}
  \resizebox{4.0cm}{!}{\includegraphics{fig_pillow_t00.eps}}&
  \resizebox{4.0cm}{!}{\includegraphics{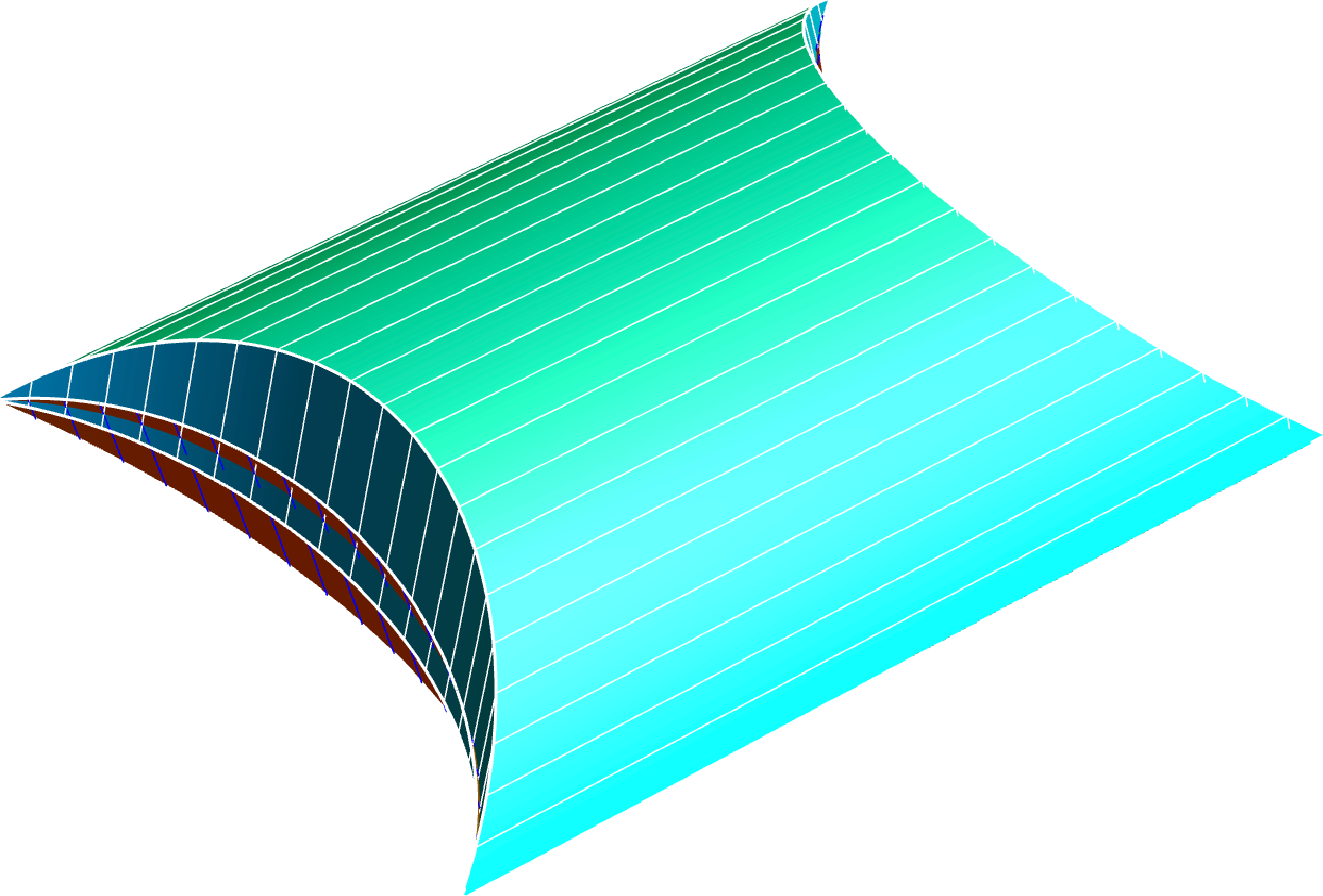}}&
  \resizebox{4.0cm}{!}{\includegraphics{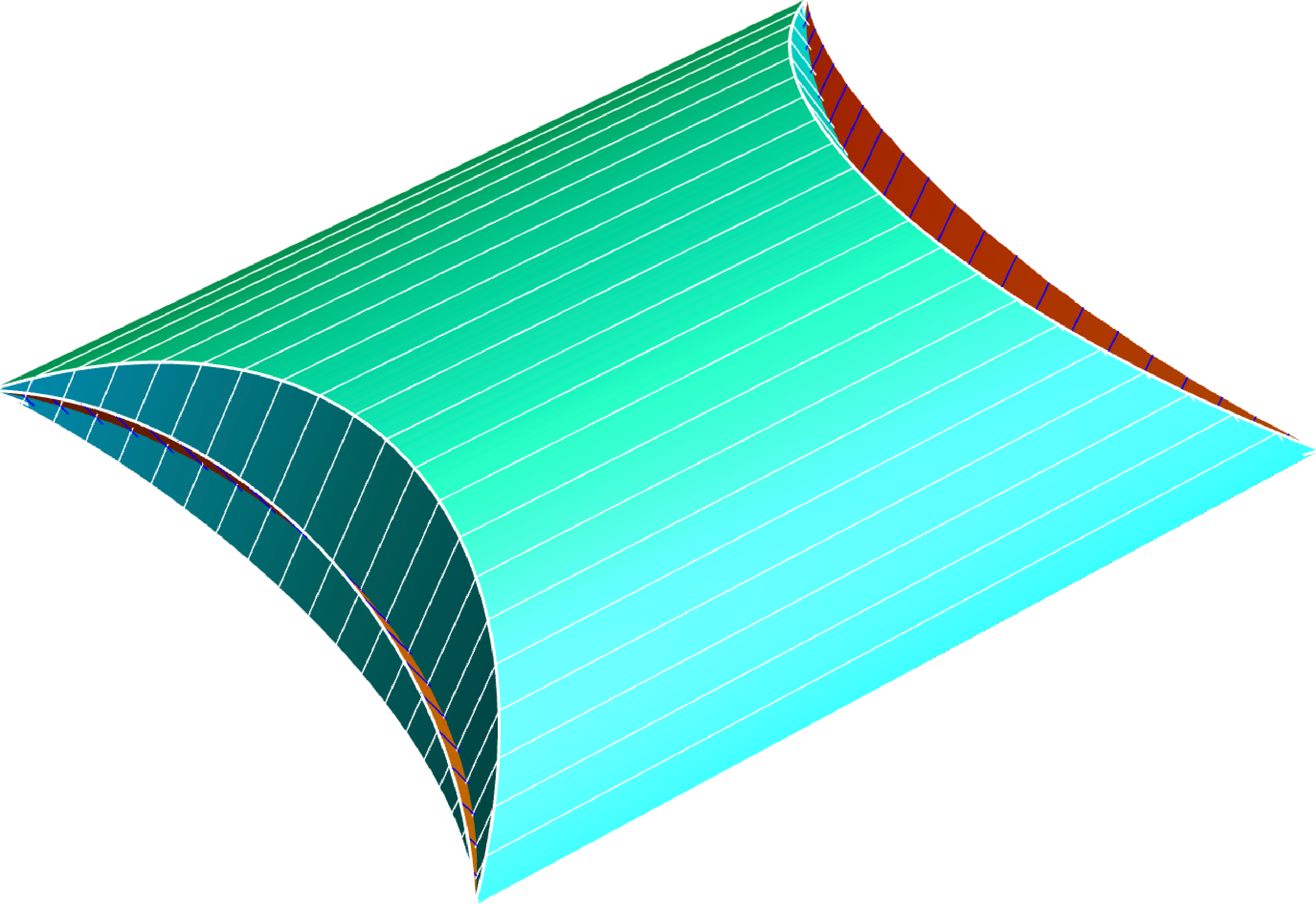}}\\
\vspace{0.5cm}
  {\footnotesize $M = M^0$}&
  {\footnotesize $M^{0.2}$}&
  {\footnotesize $M^{0.4}$}\\
  \resizebox{4.0cm}{!}{\includegraphics{fig_pillow_t06.eps}}&
  \resizebox{4.0cm}{!}{\includegraphics{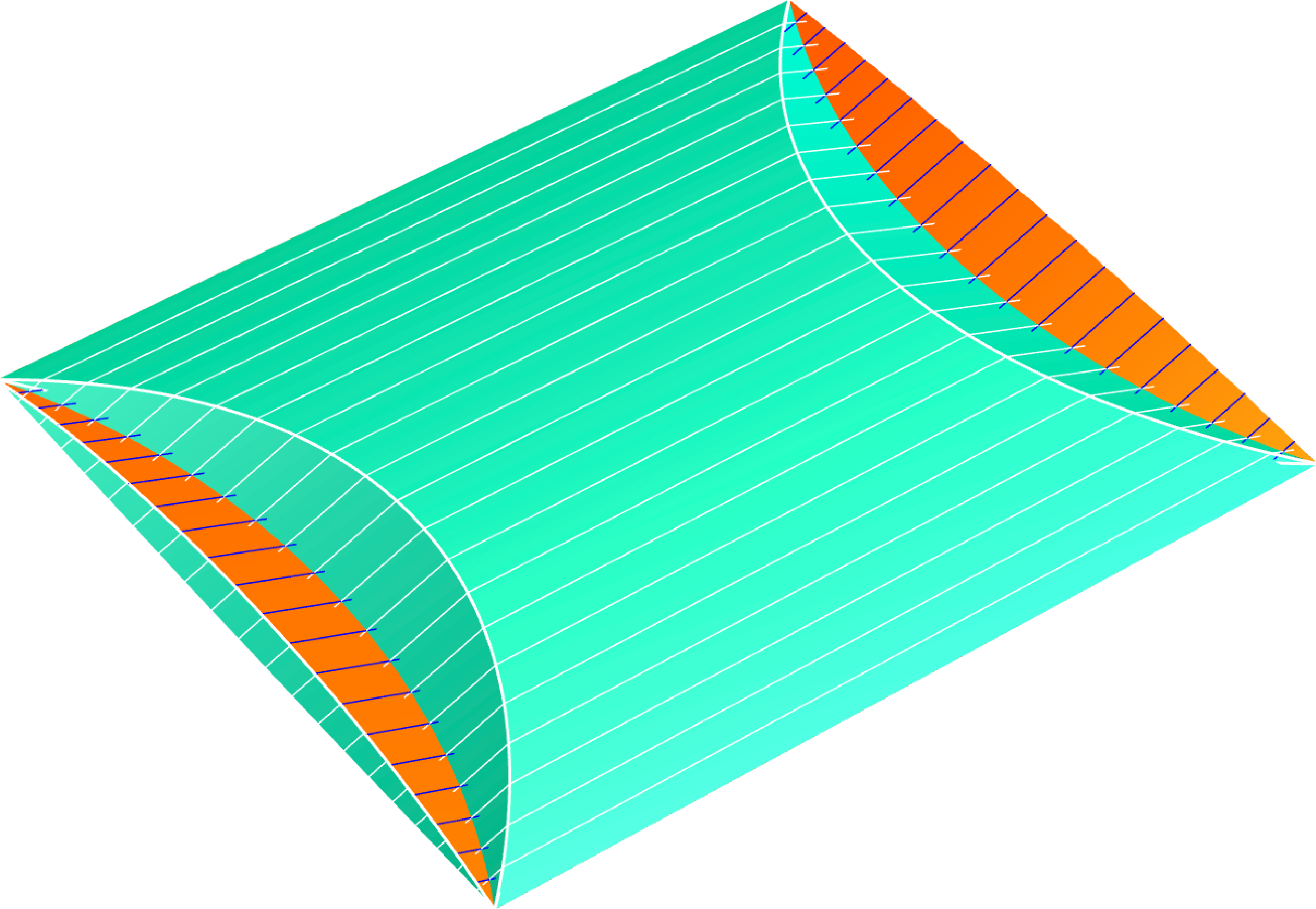}}&
  \resizebox{4.0cm}{!}{\includegraphics{fig_pillow_t10.eps}}\\
  {\footnotesize $M^{0.6}$}&
  {\footnotesize $M^{0.8}$}&
  {\footnotesize $\tilde{R} = M^1$}
\end{tabular}
\caption{The origami deformation $\{ M^t \}_{t \in [0,1]}$ of 
a pillow box $M = M^0$ to 
the double rectangle $\tilde{R} = M^1$
(see Example~\ref{ex:Koiso-deformation}).}
\label{fig:pillow-deformation-long}
\end{figure}

In Figure~\ref{fig:pillow-deformation-long},
one can observe that the upper blue part and 
the lower orange part intersect in their interiors
for $t$ other than $0$ or $1$.
In particular, $M^t$
is not homeomorphic to a $2$-sphere 
for any $t \in (0,1)$,
although both the pillow box $M^0 = M$ and 
the double rectangle $M^1 = \tilde{R}$
are homeomorphic to a sphere.
As a corollary of Theorem~\ref{thm:Koiso-deformation}, 
this phenomenon holds for general origami deformations as well:

\begin{corollary}\label{cor:isometric-deformation}
Any origami deformation necessarily changes the topology of a pillow box.
\end{corollary}

\begin{proof}
We use the notations given in Theorem \ref{thm:Koiso-deformation}.
Due to the symmetry of pillow boxes with respect to the horizontal plane $H_M$
(cf.\ Definition~\ref{def:pillow-box}), 
if an origami deformation $\{M^t\}_{t\in [0,1]}$ preserves
the topology of a pillow box, 
then for each $t\in [0,1]$, 
the horizontal end (cf.\ Definition~\ref{def:Q-origami-deformation})
$$
\Psi^t(s)
=
X^t(s,\zeta(s))
=
\vect{c}^t(s)+\zeta(s)\xi^t
\qquad
(s\in [0,L])
$$ 
of the corresponding quarter origami deformation 
$\{X^t : U \to\R^3\}_{t\in [0,1]}$
must always be contained within the horizontal plane $z=0$.

The third component of the horizontal end $\Psi^t(s)$
is written as 
$$
  \lambda(t)\zeta(s) -\zeta(s)\frac{2\lambda(t)}{1+\lambda(t)^2}
  =  -\lambda(t)\frac{1-\lambda(t)^2}{1+\lambda(t)^2}\zeta(s).
$$
Since $\lambda(t)$ is continuous,
there exists $t_1\in (0,1)$ such that $\lambda(t_1)=1/2$.
Then 
the third component of the horizontal end $\Psi^{t_1}(s)$ at $t_1$
is negative for each $s\in (0,L)$.
Hence, we obtain the assertion.
\end{proof}



\end{document}